\documentclass[12pt]{amsart}

\usepackage[left=2.54cm, right=2.54cm, top=2.54cm, bottom=2.54cm]{geometry}
\usepackage{various}
\usepackage{enumitem}

\newcommand*{\smallmajorarcs}{
  \begin{propositionE} \label{prop:smallmajorarcs}
    Let $f$ be a completely multiplicative function with $|f| \leq 1$.
    Let $W$ be a smooth function compactly supported in $(1/4, 4)$ with $|W| \leq 1$. Let $\eta \in (0, \tfrac{1}{100})$ and set $\delta = \exp( - \eta^{-4} )$.
    Let $1 \leq Q \leq N^{\delta}$ and $\mathfrak{M}_{Q, N}$ be defined as in \eqref{eq:majordef}.
    Then, for all $N > N_{0}(\eta)$,
    $$
    \int_{\mathfrak{M}_{Q, N}} \Big | \sum_{n} f(n)e(n \alpha) W \Bigl ( \frac{n}{N} \Bigr )\Big |^{2} d \alpha
    \ll \| W \|_{2,2}^2 \cdot \Big ( \frac{N}{\log(1/\eta)} + \frac{N}{\delta^{4}} \exp ( - 2 M_{f, \eta^{-2}}(Q; N) ) \Big ),
    $$
    with an absolute constant in $\ll$.
  \end{propositionE}
}

\newcommand*{\fAdefinition}{
  $$
  f_{\geq A}(p) = \begin{cases}
    1 & \text{ if } p \leq A \\
    f(p) & \text{ if } p > A
  \end{cases}.
  $$
}

\newcommand*{\majorarcsdef}{
  \begin{equation} \label{eq:majordef}
    \mathfrak{M}_{Q, N} := \bigcup_{\substack{q \leq Q \\ (a,q) = 1}} \Bigl ( \frac{a}{q} - \frac{Q}{q N}, \frac{a}{q} + \frac{Q}{q N} \Bigr ),
  \end{equation}
}

\newcommand*{\pretentiousDistanceDef}{
  $$
  M_{f, A}(Q; N) := \inf_{\substack{q \leq Q \\ \psi \text{ primitive} \\ \psi \pmod{q} \\ |t| \leq A Q / q }} \Bigl ( \sum_{Q / q \leq p \leq N \\ (p, q) = 1} \frac{1 - \Re f(p) \overline{\psi(p) p^{it}}}{p} \Bigr ).
  $$
}

\newcommand*{\reduction}{
  \begin{lemmaD} \label{le:reduction}
    Let $f$ be a multiplicative function with $|f| \leq 1$.
    Let $W$ be a smooth function compactly supported in a closed interval $I \subset (0, \infty)$. Given $M \geq 1$ let $W_{M}$ be a smooth compactly supported function such that $W_{M}(x) = 1$ for all
    $$
    x \in \bigcup_{d \leq M} d I
    $$
    where for an interval $I = [a,b]$ we write $d I := [d a, d b]$.
    For any $A$, let
    $$
    K := A^{100 \log\log A}.
    $$
    Then, for $N \geq \exp(A)$,
    \begin{align*}
      \int_{\mathfrak{M}_{Q, N}} & \Big | \sum_{n} f(n) e(n \alpha) W \Bigl ( \frac{n}{N} \Bigr ) \Big |^{2} \\ & \ll \| W \|^2_{2, 2} \cdot \Big ( K \sup_{|t| \leq A} \int_{\mathfrak{M}_{K Q, N}} \Big | \sum_{n} f_{\geq A}(n) n^{-it} e(n \alpha) W_{K} \Bigl ( \frac{n}{N} \Bigr ) \Big |^{2} d \alpha + \frac{N}{\log A} \Big ) ,
    \end{align*}
    where the implicit constant in $\ll$ is absolute.
  \end{lemmaD}
}

\newcommand*{\maincriterion}{
  \begin{propositionA} \label{prop:maincriterion}
    Take a sequence $a : \mathbb{N} \rightarrow \mathbb{C}$ with a decomposition $a = a_{1} + a_{2} + a_3$.
    Let $\mathfrak{m}$ and $\mathfrak{M}$ be two measurable sets such that $[0,1] = \mathfrak{m} \cup \mathfrak{M}$.

    Write
    \begin{equation*}
      S(\alpha) = \sum_{n \geq 1} a(n)e(n\alpha),
    \end{equation*}
    and for $i\leq 3$, write
    \begin{equation*}
      S_i(\alpha) = \sum_{n\geq 1} a_i(n)e(n\alpha).
    \end{equation*}
    
    Suppose that for some $\delta_1, \delta_2, \delta_3, \Delta > 0$ with $\delta_1 + \delta_2 + \delta_3 < 1$, the $S_i$ satisfy
    the following properties.
    \begin{enumerate}[label=(A.\arabic*)]
    \item
      $\norm{(S_1 + S_2)\mathbf{1}_{\mathfrak{M}}}_2\le\delta_1\norm{S}_{2}$; \label{i:majarc_ass}
    \item
      $\norm{S_2}_2\le\delta_2\norm{S}_2$; \label{i:s2_ass}
    \item
      $\norm{S_3}_2\le\delta_3\norm{S}_2$; \label{i:s3_ass}
    \item
      $|S_1(\alpha)|\le \Delta^{-1}N^{1/2}\norm{S}_2$ for all $\alpha\in \mathfrak{m} $. \label{i:minarc_ass}
    \end{enumerate}
    Then, we have that
    \begin{equation*}
      \norm{S}_1 \ge\frac{1}{12} \Big (1 - (\delta_1 + \delta_2 + \delta_3) \Big )^{3} \Delta N^{-1/2}\norm{S}_2.
    \end{equation*}
  \end{propositionA}
}

\newcommand*{\turankubiliuslemma}{
  \begin{lemmaB} \label{le:turankubiliuslemma}
    Let $a: \mathbb{N} \rightarrow \mathbb{C}$ be any sequence with $|a(n)| \leq 1$. Let $I \subset [1, N]$. There exists an absolute constant $N_{0}$ such that for all $N > N_{0}$,
    $$
    \int_{0}^{1} \Big | \sum_{n \geq 1} a(n) c_{2}(n; I) \Big |^{2} d \alpha \leq 4 N\Bigl ( \sum_{p \in I} \frac{1}{p} \Bigr )^{-1}
    $$
  \end{lemmaB}
}

\newcommand*{\minorarcbound}{
  \begin{lemmaC} \label{le:minorarcbound}
    Let $f: \mathbb{N} \rightarrow \mathbb{C}$ be a multiplicative function with $|f| \leq 1$. Let $W$ a smooth
    function compactly supported in $(1/4, 4)$. Let $1 \leq Q \leq N^{1/100}$. Set $I = [Q, N^{1/100}]$. Then, for $\alpha \in \mathfrak{m}_{Q, N}$,
    $$
    \sum_{n \geq 1} f(n)c_{1}(n; I) e(n \alpha) W \Bigl ( \frac{n}{N} \Bigr ) \ll   \| W \|_{1, 2} \cdot \frac{N}{\sqrt{Q}}
    $$
    with an absolute implicit constant in $\ll$.
  \end{lemmaC}
}

\newcommand*{\cpsidef}{
  $$
  c_{\psi}(n) := \sum_{\substack{x \pmod{q} \\ (x, q) = 1}} \psi(n) e \Bigl ( \frac{n x}{q} \Big).
  $$
}

\newcommand*{\halaszvariant}{
  \begin{propositionG} \label{prop:halaszvariant}
    Let $\psi$ be a primitive character of conductor $q$. Let $W$ be a smooth function compactly supported in $(0, \infty)$ with $|W| \leq 1$.
    Let $f$ be a completely multiplicative function with $|f| \leq 1$. Let $\varepsilon_{0} \in (0, 10^{-6})$ be given. Then, for $R \leq N^{\varepsilon_{0}}$,
    $T := 1 / \varepsilon_{0}$, and all $N$ sufficiently large with respect to $1/\varepsilon_{0}$,
    \begin{align} \label{eq:main}
      \sum_{\substack{\chi \pmod{r} \\ \psi \text{ induces  } \chi \\ r \leq R}} \frac{1}{\varphi(r)} \Big | \frac{1}{N} \sum_{n} f(n) c_{\chi}(n) W \Bigl ( \frac{n}{N} \Bigr ) \Big |^{2}
      \ll \frac{N^{2}}{\varepsilon_{0}^{3/2}} \exp( - 2 M_{f, \psi, T}(R; N)) + \frac{N^{2}}{\log(1/\varepsilon_0)}
    \end{align}
    where the implicit constant in $\ll$ is allowed to depend on the test function $W$ and where we set
    $$
    M_{f, \psi, T}(R; N) := \inf_{|t| \leq T} \sum_{R / q \leq p \leq N} \frac{1 - \Re f(p) \overline{\psi(p)}p^{it}}{p}
    $$
  \end{propositionG}
}

\newcommand*{\specialtheorem}{
  \begin{specialTheorem} \label{thm:specialmain}
    Let $f : \mathbb{N} \rightarrow \mathbb{C}$ be a multiplicative function with $|f(n)| \leq 1$ for all integers $n \geq 1$.
    Let $c \in (0,1)$ and $N, \Delta \geq 1/c$ be such that,
    $$
    c N \leq \sum_{n \leq N} |f(n)|^{2}
    $$
    and
    $$
    \int_{0}^{1} \Big | \sum_{n \leq N} f(n) e(n \alpha) \Big | d \alpha \leq \Delta.
    $$
    Then, there exists a real number $t$ and a primitive Dirichlet character $\chi$ of conductor $q$,
    such that $(1 + |t|) q \ll c^{-3} \Delta^{2}$ and
    $$
    \Big | \sum_{\substack{\Delta^{2}  \leq p \leq N \\ (p,q) = 1}} \frac{1 - \Re \overline{f(p)} p^{it} \chi(p)}{p} \Big | \ll \frac{1}{c}
    $$
  \end{specialTheorem}
}

\newcommand*{\maintheorem}{
  \begin{mainTheorem} \label{thm:main}

    Let $c, \varepsilon > 0$. Let $g : \mathbb{N} \rightarrow \mathbb{C}$ be a sequence and $f: \mathbb{N} \rightarrow \mathbb{C}$ a multiplicative function such that $|f(n)| \leq 1$
    for all integers $n \geq 1$. Let $\mathcal{N} := \mathcal{N}_{c, \varepsilon, g, f}$ be the set of all integers $N \geq 1$, such that,
    $$
    c N \leq \sum_{\varepsilon N \leq n \leq (1 - \varepsilon) N} |g(n)|^{2} ,
    $$
    and
    $$
    \sum_{n \leq N} |g(n) - f(n)|^{2} \leq (1 - \varepsilon) \sum_{\varepsilon N \leq n \leq (1 - \varepsilon) N} |g(n)|^{2}.
    $$
    There exists an absolute constant $K > 0$ such that if
    $N \in \mathcal{N}_{c, \varepsilon, g, f}$ and $1 / c \leq \Delta \leq N^{\gamma}$ with
    $\gamma = \exp  ( - (1 / c) \exp(\exp(K \varepsilon^{-5})) )$
    and
    $$
    \int_{0}^{1} \Big | \sum_{n \leq N} g(n) e(n \alpha) \Big | d \alpha \leq \Delta
    $$
    then there exists a real number $t$ and a primitive Dirichlet character of modulus $q$ such that
    $(1 + |t|) q \ll c^{-3} \rho(\varepsilon) \Delta^{2}$ with $\rho(\varepsilon) = \exp(\exp(K \varepsilon^{-5}))$ and
    $$
    \Big | \sum_{\substack{\Delta^{2} / q \leq p \leq N \\ (p,q) = 1}} \frac{1 - \Re \overline{f(p)} p^{it} \chi(p)}{p} \Big | \ll \rho(\varepsilon)
    $$
  \end{mainTheorem}
}

\title{$L^{1}$ means of exponential sums with multiplicative coefficients. II.}
\author{Mayank Pandey}
\address{Department of Mathematics, Princeton University, Princeton, NJ 08540, USA}
\email{mayankpandey9973@gmail.com}
\author{Maksym Radziwill}
\address{Department of Mathematics, Northwestern University, 2033 Sheridan Road, Evanston, IL, 60208, USA}
\email{maksym.radziwill@gmail.com}
\begin{document}

%% Include abstract only on the arxiv.

\begin{abstract}
  Let $f$ be a real-valued $1$-bounded multiplicative function. Suppose that the mean-value of $f^{2}$ exists,
  and
  $$
  \int_{0}^{1} \Big | \sum_{n \leq N} f(n)e^{2\pi i n \alpha} \Big | d \alpha\leq N^{o(1)}
  $$
  as $N \rightarrow \infty$,
 then there exists a quadratic character $\chi$ such that for every
  $\delta > 0$ the (logarithmic) proportion of primes $p \leq N$ such that $|f(p) - \chi(p)| < \delta$ tends to $1$ as $N \rightarrow \infty$.

  More generally we show that for all $N, \Delta \geq 1$ and $1$-bounded multiplicative functions $f$, if
  \begin{equation}\label{eq:maincond}
    \int_{0}^{1} \Big | \sum_{n \leq N} f(n) e^{2\pi i n \alpha} \Big | d \alpha \leq \Delta
  \end{equation}
  and the $L^{2}$ norm of $f$ over $[1, N]$ is $\geq N / 100$,
  then $f$ pretends to be a multiplicative character of conductor $\leq \Delta^{2}$ on primes
  in $[\Delta^{2}, N]$. We highlight that the result is uniform in $f$, $N$ and $\Delta$ and sharp as far as the
  size of the conductor goes. Moreover, the restriction to primes $p \in [\Delta^{2}, N]$ turns out to be sharp
  in a suitably generalized version of this result, concerning sequences $f$ that are close $1\%$ of the
  time to multiplicative functions.

\end{abstract}
\maketitle
%\tableofcontents
\section{Introduction}
It was conjectured by Littlewood that for any finite $S \subset \mathbb{N}$,
\begin{equation} \label{eq:little}
\int_{0}^{1} \Big | \sum_{n \in S} e^{2\pi i n \alpha} \Big | d \alpha > c \log |S|
\end{equation}
with $c > 0$ an absolute constant. Littlewood also conjectured that as $S$ varies among subsets of $\mathbb{N}$ of cardinality $N$ the expression \eqref{eq:little} is minimized when  $S = [1, N]$. The former conjecture was resolved by Konyagin~\cite{littlewood2} and McGehee-Pigno-Smith~\cite{littlewood}, while the latter remains open. The work of McGehee-Pigno-Smith \cite{littlewood} implies that for any sequence $a : \mathbb{N} \rightarrow \mathbb{C}$ such that $|a(n)| \geq 1$ for all $n \geq 1$,
$$
\int_{0}^{1} \Big | \sum_{n \leq N} a(n) e(n \alpha) \Big | d \alpha \geq c \log N,
$$
with $c > 0$ an absolute constant.
It is expected that
\begin{equation} \label{eq:l1norm}
\int_{0}^{1} \Big | \sum_{n \leq N} a(n) e(n \alpha) \Big | d \alpha
\end{equation}
is small when $a(n)$ is in some sense ``additive''. For example, it is conjectured \cite{Green} that if $a(n)$ is the indicator function of a set $S \subset \mathbb{N}$ and there exists a $K > 0$ such that \eqref{eq:l1norm} is $\leq K \log N$
for all $N \geq 2$
then there exists arithmetic progressions $P_{1}, \ldots, P_{J}$ with $J = O_{K}(1)$ such that,
$$
\lim_{N \rightarrow \infty }\frac{1}{N} \sum_{n \leq N} \Big | \mathbf{1}_{S}(n) - \sum_{1 \leq i \leq J} c_{i} \mathbf{1}_{P_{i}}(n) \Big | = 0.
$$
with $c_{i} \in \{-1,1\}$.
It is reasonable to expect that multiplicative functions such as the M\"obius function (denoted $\mu$) or the Liouville function (denoted $\lambda$) contain no tangible additive structure. As a result when $a \in \{\mu, \lambda\}$ the $L^1$ norm \eqref{eq:l1norm} should be as large as possible, that is $\gg N^{1/2 - o(1)}$.
Lower bounds for $L^1$ norms \eqref{eq:l1norm} with $a \equiv \mu$ or $a \equiv \lambda$ are the subject of works of Balog-Perelli \cite{BalogPerelli} and respectively Balog-Ruzsa \cite{BalogRuzsa1, BalogRuzsa2}. Recently the authors showed that in both cases the left-hand side of \eqref{eq:l1norm} is at least $\gg N^{1/4 - o(1)}$. The proof of this lower bound crucially uses the connection of $\mu$ and $\lambda$ with zeros of $L$-functions, and we note that even on the assumption of the Generalized Riemann Hypothesis stronger lower bounds are not currently known. Besides the Liouville and M\"obius function, other multiplicative functions such as the indicator function of $k$-free numbers \cite{etal} or coefficients of modular forms \cite{Pandey} have received attention, with more complete results.

More generally one would like to assert that \eqref{eq:l1norm} is large (i.e $\gg N^{c}$ for some constant $c > 0$) whenever $f$ is a multiplicative function. This is however false. The simplest counterexample is of course $f(n) = 1$. Moreover $f(n) = \chi(n)$ with $\chi$ a quadratic character provides another set of counterexamples. Our first result shows that for real-valued $f$ those are the only obstructions.
\begin{corollary} \label{cor:first}
  Let $f : \mathbb{N} \rightarrow \mathbb{R}$ be a $1$-bounded multiplicative function such that,
  $$
  \liminf_{N \rightarrow \infty} \frac{1}{N} \sum_{n \leq N} f(n)^{2} > 0.
  $$
  Suppose that there exists a $\psi(N) \rightarrow 0$ arbitrarily slowly with $N \rightarrow \infty$ such that,
  $$
  \int_{0}^{1} \Big | \sum_{n \leq N} f(n) e(n \alpha) \Big | d \alpha \leq N^{\psi(N)}.
  $$
  Then, there exists a quadratic Dirichlet character $\chi$ such that,
  $$
  \sum_{p \leq N} \frac{1 - f(p) \chi(p)}{p} = o \Big ( \sum_{p \leq N} \frac{1}{p} \Big )
  $$
  as $N \rightarrow \infty$. That is for any $\delta > 0$, the proportion of $p$ (in a logarithmic sense) such that
  $|\chi(p) - f(p)| < \delta$ tends to one.
\end{corollary}
To gauge the strength of the result, notice that until the recent work of the authors, already for $f = \mu$ it was not ruled that the $L^{1}$ norm is less than $N^{1 / \log\log\log N}$ for all $N$.

Corollary \ref{cor:first} concerns only real-valued one-bounded multiplicative function, we will now discuss the situation for complex-valued multiplicative functions. Let $\chi$ be a primitive Dirichlet character of modulus $q$ and let $t \in \mathbb{R}$ be a real-number. A multiplicative function $f$ defined by setting $f(n) = \chi(n) n^{it}$ for all $n \geq 1$ is called a \textit{multiplicative character} and the real number $q (1 + |t|)$ is known as the \textit{conductor} of $f$. It is not hard to show (using Poisson summation) that if $f$ is a multiplicative character of conductor $\leq \Delta^{2}$ and $W$ is a smooth function compactly supported in $(0,1)$, then,
%$$
%\int_{0}^{1} \Big | \sum_{n \leq N} f(n) e(n \alpha) \Big | d \alpha \ll \Delta \log N.
%$$\marginpar{Check the $\log N$ term}
%and if we pick a smooth function $W$ compactly supported in $[0,1]$, then,
$$
\int_{0}^{1} \Big | \sum_{n \leq N} f(n) e(n \alpha) W \Big ( \frac{n}{N} \Big ) \Big | d \alpha \ll \Delta.
$$
%There is an additional class of non-pretentious multiplicative functions $f$ for which the $L^{1}$ norm \eqref{eq:l1norm} is %small. An representative example is,
%$$
%f(p) = \begin{cases}
%  1 & \text{ if } p < W \\
%  \chi_{-4}(p) & \text{ if } p > W
%  \end{cases}
%  $$
%  with $W = \log\log N$, say.
%  This function is not pretentious, and yet has small $L^{1}$ norm.
Our main theorem is uniform in $\Delta$ and $N$ and classifies complex valued $1$-bounded multiplicative functions $f$ such that
\begin{equation} \label{eq:l1norm'}
\int_{0}^{1} \Big | \sum_{n \leq N} f(n) e(n \alpha) \Big | d \alpha \leq \Delta.
\end{equation}
The upshot is that \eqref{eq:l1norm'} can happen only if $f$ is at least in part overlapping with a multiplicative character of conductor $\ll \Delta^2$. We expect that in applications the uniformity in $\Delta, N$ and $f$ of our theorem will be useful. In practice assumptions that are valid for all $N$ large enough, such as the one appearing in Corollary \ref{cor:first} are seldom available.
% Our main theorem shows that outside of these obstructions the $L^{1}$ norm of a multiplicative function has to be large. Compared to the Corollary our Main Theorem holds for any fixed scale $N$ (i.e we do not require that the $L^{1}$ norm is small for all $N \geq 1$). This should be more useful in applications.
\specialtheorem
\begin{remark}
  Notice that the conclusion is vacuous if $\Delta \geq N^{\eta}$ for some $\eta = \eta(c) > 0$. We could have therefore introduced the additional assumption $1 \leq \Delta \leq N^{\eta}$ without loss of generality. 
\end{remark}
The conclusion of our theorem concerns only primes larger than $\Delta^{2}$. At first this is a little puzzling.
If one generalizes our result just a little bit (as we will do now), then this restriction turns out to be optimal, thus explicating the origin of this condition.
%To date there have been no results classifying small $L^{1}$ norms for general multiplicative functions. Our main result is a first result in this direction. In fact, we show more, showing a lower
%bound for the $L^1$-norm for any sequence that resembles a nonpretentious multiplicative function ``1\%'' of the time. More precisely, we show the following.
%\marginpar{Check if there is a $(p,q) = 1$ condition in pretentious distance.}
\maintheorem
\begin{remark}
  Notice that the conclusion of this theorem concerns primes in the interval $[\Delta^{2} / q; N]$ rather than $[\Delta^{2}; N]$, this comes at the expense of adding an explicit assumption that $\Delta \leq N^{\gamma}$ with $\gamma$ depending on $\varepsilon$ and $c$. Moreover pretentious distance is bounded solely in terms of $\varepsilon$, this comes at the expense of adding the assumption $1 / c \leq \Delta$.
\end{remark}
We note that Theorem A follows from Theorem B.
We record here another application of Theorem B.
% For another application of Theorem B, consider $g(n) = \sum_{i \leq R} c_{i} f_{i}(n)$ with each $f_{i}$ a $1$-bounded multiplicative function such that for $i \neq j$,
%$$
%\lim_{N \rightarrow \infty} \frac{1}{N} \sum_{n \leq N} f_{i}(n) \overline{f_{j}(n)} = 0, \ \liminf_{N \rightarrow \infty} \frac{1}{N} \sum_{n \leq N} |f_{i}(n)|^{2} \geq c.
%$$
%Theorem B implies for all $N$ sufficiently large with respect to $R$ that if the $L^{1}$ norm of a trigonometric polynomial with coefficients $g(n)$ is smaller than $\Delta$ then each $f_{i}$ pretends to be a multiplicative character of conductor $\ll_{R, c} \Delta^{2}$ on primes $p \in [\Delta^{2}, N]$. We also record the following Corollary.
\begin{corollary} \label{cor:second}
  Let $f_{1}, \ldots, f_{R} : \mathbb{N} \rightarrow \mathbb{R}$ be a $1$-bounded multiplicative functions such that,
  $$
  \liminf_{N \rightarrow \infty} \frac{1}{N} \sum_{n \leq N} f_{i}(n)^{2} > 0 \ , \ \lim_{N \rightarrow \infty} \frac{1}{N} \sum_{n \leq N} f_{i}(n) f_{j}(n) = 0.
  $$
  Let
  $$
  g(n) = \sum_{i \leq R} c_{i} f_{i}(n)
  $$
  with $c_{i} \in \mathbb{C}$ non-zero.
  Let $\psi(N) \rightarrow 0$ arbitrarily slowly with $N \rightarrow \infty$. Suppose that for all $N \geq 1$
  $$
  \int_{0}^{1} \Big | \sum_{n \leq N} g(n) e(n \alpha) \Big | d \alpha \leq N^{\psi(N)},
  $$
  Then, there exists primitive quadratic Dirichlet characters $\chi_{i}$ such that,
  $$
  \lim_{N \rightarrow \infty} \Big ( \sum_{p \leq N} \frac{1}{p} \Big )^{-1} \sum_{p \leq N} \frac{1 - f_{i}(p) \chi_{i}(p)}{p} = 0
  $$
  for each $1 \leq i \leq R$. That is for any $1 \leq i \leq R$ and $\delta > 0$, the proportion of $p$ (in a logarithmic sense) such that $|\chi_{i}(p) - f_{i}(p)| < \delta$ tends to one.
\end{corollary}
We do not prove Corollary \ref{cor:second}, the proof is similar to the proof of Corollary \ref{cor:first}.

In general, the conclusion of Theorem B is sharp, i.e we cannot conclude any statement about primes smaller than $\Delta^{2}$ in the pretentious distance. In particular to further improve Theorem A we need a method of proof that uses more than just the closeness in $L^{2}$ of $g$ to a multiplicative function $f$. To illustrate the cases in which Theorem B is sharp we focus on the case $q = 1$. Consider any multiplicative function $f$ with $|f| \leq 1$ and such that $f(p) = 1$ for $p \geq \Delta^{1 / 2A}$. Let $L := \Delta^{1/2A}$. Let $g$ be any reasonable sieve weight approximating $f$ in $L^{2}$ and having level of distribution $\sqrt{\Delta}$. It is then the case that
$$
\sum_{n \leq N} |g(n) - f(n)|^{2} \leq \frac{1}{2} \sum_{n \leq N} |g(n)|^{2}
$$
provided that $A$ is taken large enough but fixed. At the same time the $L^{1}$ norm of $\sum_{n \leq N} g(n) e(n \alpha)$ is small because $g$ is a sieve weight, in fact easily seen to be $\ll \Delta$. Since the function $f$ was arbitrarily defined on primes $\leq L$ we conclude that the hypothesis of our theorem is not strong enough to make any conclusion about the behavior of $f$ on primes $\leq L$. Since
$$
\sum_{L \leq p \leq \Delta^{2}} \frac{1}{p} \ll \log A = O(1)
$$
we see that we cannot assert anything about primes $\leq \Delta^{2}$ in the pretentious distance.

We now present a brief description of our proof.
A more detailed description can be found in the next section. Consider the special case when $g = f$ is a multiplicative function.
The guiding idea in our proof (going back to Balog-Perelli) is that if the $L^{1}$ norm of
\begin{equation} \label{eq:trig1}
\alpha \mapsto \sum_{n \leq N} f(n) e(n \alpha)
\end{equation}
is small, then this is due to the majority of the
mass of the trigonometric polynomial \eqref{eq:trig1}
accumulating in a set of $\alpha$ of small measure.
When $f$ is multiplicative it is reasonable to posit that this small measure set of $\alpha$ corresponds to the major arcs
\majorarcsdef
with $Q \asymp \Delta^{2}$. This is in fact what we show: if the $L^{1}$ norm of \eqref{eq:trig1} is small, then
$$
N \ll \int_{\mathfrak{M}_{Q, N}} \Big | \sum_{n \leq N} f(n) e(n \alpha) \Big |^{2} d \alpha \ , \ Q \asymp \Delta^{2}.
$$
We then conclude by showing that the latter implies that
$$
\sum_{\Delta^{2} / q \leq p \leq N} \frac{1 - \Re \overline{f(p)} p^{it} \psi(p)}{p} = O(1).
$$
for some primitive character $\psi$ of conductor $q$ and a $t \in \mathbb{R}$ with $q (1 + |t|) \ll \Delta^{2}$.
The latter conclusion is optimal: If $f$ is a multiplicative function with $f(p) = 1$ for $Q \leq p \leq N$ and $Q$ less than a small power of $N$, then one can show, by approximating $f$ by a sieve weight and then applying Poisson summation, that
$$
\int_{\mathfrak{M}_{Q, N}} \Big | \sum_{n \leq N} f(n) e(n \alpha) \Big |^{2} d \alpha \asymp \sum_{n \leq N} |f(n)|^{2}.
$$
The same limitation arises also when $f$ is equal to a character of small conductor on large primes.
%It thus remains an open question if the conclusion in our first theorem can be reduced to say something about primes less than $\Delta^{\varepsilon}$.

%\subsection{Corollaries}

%We list in this subsection some corollaries of the main theorem. 
%\begin{corollary}
%  Fix some $\delta > 0$ and suppose that $N$ is a sufficiently large integer.
%  Let $g : \mathbb{N} \to \set{\pm 1}$ be any sequence such that for some multiplicative $f : \mathbb{N} \to\set{\pm 1}$,
%  \begin{equation*}
%    \sum_{n \leq N\\ f(n) = g(n)} 1 \geq \biggl(\frac{1}{2} + \delta\biggr)N.
%  \end{equation*}
%  Then, there exists $\eta_0 = \eta_0(\delta)$ depending only on $\delta$ such that if
%  \begin{equation*}
%    \int_{0}^1 \biggl|\sum_{n\leq N} g(n)e(n\alpha)\biggr|d\alpha  \leq  \Delta
%  \end{equation*}
%  for some $\Delta \leq N^{\eta_{0}}$, then there exists a primitive character $\psi$ of conductor $q$
%  and a $t \in \mathbb{R}$ with $q (1 + |t|) \ll_{\delta} \Delta^{2}$ such that
%  $$
%  \sum_{\Delta^{2} / q \leq p \leq N \\ (p, q) = 1} \frac{1 - \Re f(p) p^{it} \psi(p)}{p} = O_{\delta}(1).
%  $$
%\end{corollary}
%This corollary is an immediate consequence of the following more general corollary for general $f, g$.
\subsection{Generalizations and other consequences}

Our main result can be extended to any $L^{p}$ norm with $p \in (0,2)$. It seems likely that our methods can extend to the case of divisor bounded multiplicative functions for which efficient sieve majorants exists.
Finally, one can use Proposition A along with Voronoi summation and multiplicativity to show that if $\lambda_\pi(n_1,\dots,n_d)$
are the Fourier coefficients of $\pi$, a $\GL_d$ cuspidal automorphic representation, then
\begin{equation*}
\int_0^1 \biggl| \sum_{n\le N} \lambda_\pi(n,1,\dots,1) e(n\alpha) \biggr|d\alpha\gg N^{\frac{1}{d + 3/2} - o(1)}.
\end{equation*}
This is new for $d\ge 4$, and is the first nontrivial power-saving bound known for the exponential sum in those cases.
We leave the details of this to the interested reader. 
\section*{Acknowledgments}
MR acknowledges support of NSF grant DMS-2401106. The authors would like to thank Andrew Granville and Dimitris Koukoulopoulos for their interest and comments.
% It is at first surprising that we do not appear to say much about primes $\leq \Delta^{2}$. As we will explain shortly this is the optimal conclusion for our method of proof (see Proposition A). It is however reasonable to conjecture that with additional ideas, it should be possible to also say something about primes $\leq \Delta^{2}$. We caution the reader that however nothing else can be said by pushing further the methods of this paper.

% The simplest interpretation of this result is that if the $L^{1}$ norm is small (i.e $\leq \Delta$) then the multiplicative function $f$ pretends to be a multiplicative character of total conductor $\ll \Delta^{2}$ on the primes larger than $\Delta^{2}$. This is essentially optimal since a Dirichlet character of conductor $\Delta^{2}$ has $L^{1}$ norm $\ll \Delta \log \Delta$\footnote{Notice that if the sums over $n$ are smoothed then the $L^{1}$ norm is $\asymp \Delta$ and then our result is sharp on the nose; in the proof we in fact work with smoothed sums throughout}. Essentially the same result should be true for any $L^{p}$ norm with $p \in (0, 2)$.

% It might appear a little unusual that the conclusion concerns only small primes.
% \begin{conjecture}
%   Run the de la Breteche-Granville heuristics a la Hardy-Littlewood to make a conjecture giving the optimal bound. In particular the conjecture might make it clear that only the large primes are important... (or construct an explicit counterexample)
% \end{conjecture}

\section{Notation}
For any measurable $1$-periodic function $f : \mathbb{R} \rightarrow \mathbb{C}$ and any $p \geq 1$ we will write
$$
\|f \|_{p} := \Big ( \int_{0}^{1} |f(t)|^{p} dt \Big )^{1/p}.
$$
For a smooth function $f : \mathbb{R} \rightarrow \mathbb{C}$ we define for $p \geq 1$ and $r \geq 1$,
$$
\| f \|_{p, r} = 1 + \sum_{j = 0}^{r} \Big ( \int_{\mathbb{R}} |f^{(j)}(x)|^{p} dx \Big )^{1/p}.
$$
Notice that this is not the usual Sobolev norm because of the inclusion of the term $+1$. This allows 
for inequality of the type $\| W \|_{1,2} \ll \| W \|_{2,2}^2$.

\section{Outline of the argument}
The proof of our main theorem is based on the following criterion.
We have made no effort to optimize the dependence on $\delta_1, \delta_2, \delta_3$ in the final bound.
\maincriterion

Given an interval $I$ we define
$$
c_{1}(n; I) = \Bigl ( \sum_{p \in I} \frac{1}{p} \Bigr )^{-1} \sum_{\substack{p | n \\ p \in I}} 1,
$$
and
$$
c_{2}(n; I) = \Bigl ( \sum_{p \in I} \frac{1}{p} \Bigr )^{-1} \Bigl ( \sum_{p \in I} \frac{1}{p} - \sum_{\substack{p | n \\ p \in I}} 1  \Bigr )
$$
Since $c_{1} + c_{2} = 1$, any sequence $a : \mathbb{N} \rightarrow \mathbb{C}$ can be decomposed as
$$
a(n) = a(n)c_{1}(n) + a(n) c_{2}(n).
$$
By a variant of Turan-Kubilius we know that $a(n)c_{2}(n)$ has to be small on average in the following sense:

\turankubiliuslemma

\begin{proof}
  Since $|a| \leq 1$ we have
  $$
  \int_{0}^{1} \Big | \sum_{n \geq 1} a(n) c_{2}(n) e(n \alpha) \Big |^{2} d\alpha = \sum_{n \geq 1} |a(n) c_{2}(n)|^{2} \leq \sum_{n \geq 1} |c_{2}(n)|^{2}.
  $$
  By the Turan-Kubilius inequality, the above is
  $$
  \leq 4 N\Bigl ( \sum_{p \in I} \frac{1}{p} \Bigr )^{-1}
  $$
  for all $N$ sufficiently large.
\end{proof}

We define our major arcs to be
\majorarcsdef
and we also set
$$
\mathfrak{m}_{Q, N} := [0,1] \backslash \mathfrak{M}_{Q, N}.
$$
We then have the following standard bound for bilinear forms.
\minorarcbound
% \input{minorArcBoundProof.tex}
% \marginpar{Thanks to the reduction to completely multiplicative $f$ we can now remove the condition $p \| n$ in Turan-Kubilius and other things of this sort in the proof, this is very minor and does not affect correctness right now, so defer it for later}
If $W$ is a smooth compactly supported function with $\widehat{W}(0) \neq 0$, $\alpha = 1 / (Q + 1) \in \mathfrak{m}_{Q, N}$ and $f(n) = \chi(n)$ a primitive Dirichlet character with conductor $Q + 1$, then, by Poisson summation,
$$
\sum_{n} f(n) e(n \alpha) W \Big ( \frac{n}{N} \Big ) \asymp \frac{N}{\sqrt{Q}}.
$$
In addition if $Q + 1$ is a Siegel zero modulus then the presence of the factor $c_{1}(n; I)$ does not yield additional cancellations. Thus, given the current state of knowledge our bound in \eqref{le:minorarcbound} is optimal.
We now come to bounding the integral over major arcs. We first notice that we can assume without loss of generality that the multiplicative function $f$ is completely multiplicative. Give a multiplicative function $f$ we define for every $A \geq 1$ a completely multiplicative function $f_{\geq A}$ given by
\fAdefinition
We then have the following result which amounts to a form of presieving.
\reduction
In order to state our bound for the integral over the major arcs we need to introduce a notion of pretentious distance
\pretentiousDistanceDef
We then have the following result.
\smallmajorarcs
% This result could be improved by a more careful argument, essentially replacing $\delta$ by $\eta$ and $(\log 1 / \eta)^{-1}$ in the bound by $\eta$. Such an improvement could be effected by using Ramar\'e's identity instead of Turan-Kubilius and by taking out three small prime factors all at once, instead of taking them out in consecutive stages as we do.
With all of these lemmas we are now ready to give a proof of our main theorem.
\begin{proof}[Proof of Main Theorem]
  By assumption,
  $$
  \sum_{\varepsilon N \leq n \leq (1 - \varepsilon) N} |g(n)|^{2} \geq c N.
  $$
  To each $\varepsilon > 0$ we associate a function $W_{\varepsilon}$ which we require to be
  \begin{enumerate}
    \item smooth, non identically zero
    \item compactly supported in $[\varepsilon / 2,1 - \varepsilon / 2]$
    \item $W_{\varepsilon}(x) = 1$ for $\varepsilon < x < 1 - \varepsilon$
    \item $0\leq W(x) \leq 1$ for all $x \in \mathbb{R}$
    \item for all $j \geq 1$ and $x \in \mathbb{R}$, $|W^{(j)}(x)| \leq c_{j} \varepsilon^{-j}$, with $c_{j}$ depending only on $j$.
  \end{enumerate}
  In this way, if we set
  $$
  a(n) := g(n) W_{\varepsilon} \Bigl ( \frac{n}{N} \Bigr ),
  $$
  we still have
  \begin{equation}
    \label{eq:a_l2_lower}
    \sum_{n} |a(n)|^{2} \geq \sum_{\varepsilon N \leq n \leq (1 - \varepsilon) N} |g(n)|^{2} \geq c N.
  \end{equation}
  From now on, we write $W = W_{\varepsilon}$. First by a simple harmonic analysis argument (see e.g \cite[Lemma 2]{PandeyRadz}),
  \begin{equation} \label{eq:above}
    \int_{0}^{1} \Big | \sum_{n} g(n) e(n \alpha) W \Bigl ( \frac{n}{N} \Bigr ) \Big | d \alpha \ll \frac{1}{\varepsilon} \int_{0}^{1} \Big | \sum_{n \leq N} g(n) e(n \alpha) \Big | d \alpha.
  \end{equation}
  Let $I = [Q, N^{1/100}]$ with a $Q \leq N^{1/100}$. Let $\mathfrak{M} := \mathfrak{M}_{Q, N}$ with $\mathfrak{M}_{Q, N} $ as defined in \eqref{eq:majordef} and $\mathfrak{m} := [0,1] \backslash \mathfrak{M}$.
  Recall the hypothesis of the main theorem that
  $$
  \int_{0}^{1} \Big | \sum_{n \leq N} g(n) e(n \alpha) \Big | d \alpha \leq \Delta.
  $$
  Then, (\ref{eq:above}) implies that
  $$
  \int_{0}^{1} \Big | \sum_{n} g(n) e(n \alpha) W \Bigl ( \frac{n}{N} \Big) \Big | d \alpha\le
  \Delta_1\Bigl ( \frac{1}{N} \sum_{n} |a(n)|^{2} \Bigr )^{1/2}.
  $$
  where $\Delta_1 =c_2 \Delta  / (\sqrt{c} \varepsilon)$ for some absolute constant $c_2 > 0$.
  Now, decompose $a = a_1 + a_2 + a_3$ with
  \begin{align*}
    a_1(n) &= f(n)c_1(n; I)W\pfrc{n}{N}, \\
    a_2(n) &= f(n) c_2(n; I)W\pfrc{n}{N}, \\
    a_3(n) &= (g(n) - f(n))W\pfrc{n}{N}.
  \end{align*}
  We prepare to apply Proposition A.
  Write
  \begin{equation*}
    S_i(\alpha) = \sum_n a_i(n)e(n\alpha).
  \end{equation*}
  By assumption,
  \begin{align}
    \label{eq:a3_l2_upper1}
    \|S_{3}\|_{2}^{2} = \sum_{n} |a_{3}(n)|^2 & \leq \sum_{n \leq N} |g(n) - f(n)|^{2} \\ \nonumber & \leq ( 1 - \varepsilon ) \sum_{\varepsilon N \leq n \leq (1 - \varepsilon) N} |g(n)|^{2} \\ \nonumber & \leq (1 - \varepsilon) \sum_{n \leq N} |a(n)|^{2} = (1 - \varepsilon) \| S \|_{2}^{2}
  \end{align}
  And thus, since $\sqrt{1 - \varepsilon} \leq 1 - \varepsilon / 2$,
  \begin{equation}
    \label{eq:s3_upper}
  \| S_{3} \|_{2} \leq \Big ( 1 - \frac{\varepsilon}{2} \Big ) \| S \|_{2}.
  \end{equation}
  Also, Lemma B implies that if $Q$ is chosen so that
  \begin{equation} \label{eq:reqOne}
    \sum_{p \in [Q, N^{1/100}]} \frac{1}{p} > \frac{4 \cdot 16 / \varepsilon^{2}}{c}
  \end{equation}
  we have,
  \begin{equation}
    \sum_{n} |a_2(n)|^2\le \frac{\varepsilon^{2}}{16} \cdot c N \le \frac{\varepsilon^{2}}{16} \sum_n |a(n)|^2.
  \end{equation}
  Thus, by Plancherel, we have
  \begin{equation}
    \label{eq:s2_upper}
    \norm{S_2}_2\leq \frac{\varepsilon}{4} \norm{S}_2.
  \end{equation} 
  Let $K$ be a smooth function with $0 \leq K \leq 1$ compactly supported in $(1/4, 4)$ and such that,
  $$
  \sum_{k} K \Big ( \frac{n}{2^{k}} \Big ) = 1
  $$
  for every integer $n$.
  Lemma C yields the bound, for each $\varepsilon N \leq M \leq N$,
  \begin{equation}
    \sup_{\alpha\in \mathfrak{m} }\bigg|\sum_n a_1(n) K \Big ( \frac{n}{M} \Big ) e(n\alpha)\biggr| \le \frac{C_{1} \varepsilon^{-2} M}{\sqrt{Q}}%\le \frac{C_{1} \varepsilon^{-2}}{\sqrt{cQ}}N^{1/2}\norm{S}_2
  \end{equation}
  with $C_{1}$ an absolute constant, and thus,
  \begin{equation} \label{eq:minarc_bd_final}
  \sup_{\alpha \in \mathfrak{m}} \Big | \sum_{n} a_{1}(n) e(n \alpha) \Big | \le \frac{2 C_{1} \varepsilon^{-2} N}{\sqrt{Q}} \le
  \frac{2 C_{1} \varepsilon^{-2}}{\sqrt{c Q}} N^{1/2} \cdot \norm{S}_{2}.
  \end{equation}
%  For any $Q \geq K(\varepsilon)^{2} \Delta_{1}^{2} / (1024 \cdot c^{2} \varepsilon^{6})$, the above bound implies,
%  \begin{equation} \label{eq:minarc_bd_final}
%  \sup_{\alpha \in \mathfrak{m}} \Big | \sum_{n} a_{1}(n) e(n \alpha) \Big | \leq \frac{32 \varepsilon^{3}}{\sqrt{c}}% \frac{N^{1/2}}{\Delta_{1}} \cdot \| S \|^{2}_{2}
% \end{equation}
  By Lemma D applied for $A = \exp(C_{2} \varepsilon^{-4})$, $K = \exp(C_{3} \varepsilon^{-5}) \geq A^{100 \log\log A}$, and $C_{2}, C_{3} > 0$ absolute constants,
  \begin{multline}
    \int_{\mathfrak{M} } |S_1(\alpha) + S_2(\alpha)|^2d\alpha = \int_{\mathfrak{M} } \biggl|\sum_n f(n)W\pfrc{n}{N}\biggr|^2d\alpha \\
    \le K \sup_{|t|\le A}\int_{\mathfrak{M}_{K Q, N}}\biggl|\sum_{n}f_{\ge A}(n)n^{-it}e(n\alpha)W_0\pfrc{n}{N}\biggr|^2\,d \alpha + \Big ( \frac{\varepsilon}{64} \Big )^{2} \sum_{n}|a(n)|^2
  \end{multline}
  and $W_0$ a smooth function compactly supported on a closed interval $I \subset [\varepsilon/4, 2K]$, equal to $1$ on $[\varepsilon / 2, K]$ and such that
  $$
  W_{0}^{(j)}(x) \leq c'_{j} \varepsilon^{-j}
  $$
  for some constants $c'_{j}$ depending only on $j$.
  We introduce again a partition of unity, and use Cauchy-Schwarz to conclude that the left-hand side of the above expression is
  $$
  \leq K^{2} \sum_{ \varepsilon N / 4  \leq M = 2^{k} \leq 2 K N} \sup_{|t| \leq A} \int_{\mathfrak{M}_{K Q, N}} \Big | \sum_{n} f_{\ge A}(n) n^{-it} e(n \alpha) W_{0} \Big ( \frac{n}{N} \Big ) K \Big ( \frac{n}{M} \Big ) \Big |^{2} d \alpha
  $$
  Given $\eta \in (0, \tfrac{1}{100})$, we can apply Proposition E provided that
  $K Q\le N^{\delta}$ with $\delta = \exp(-\eta^{-4})$,
  \begin{multline}
    K^{2} \sum_{\varepsilon N / 4 \leq M = 2^{k} \leq 2 K N}  \sup_{|t|\le A}\int_{\mathfrak{M}_{K Q, N}}\biggl|\sum_{n}f_{\ge A}(n)n^{-it}e(n\alpha)W_0\pfrc{n}{N}\biggr|^2\,d \alpha\\
    \leq K' \Big (\frac{1}{\log(1/\eta)} + \frac{1}{\delta^4}\exp(-2M_{f_{\geq A}, \eta^{-2}}(K Q; N)) \Big )\sum_{n} |a(n)|^2.
  \end{multline}
  for some $K' = \exp(C_{4} \varepsilon^{-5})$ with $C_{4} > 0$ an absolute constant. Here after applying Proposition E we resummed the partiton of unity, and used that $M_{f_{\geq A}, \eta^{-2}}(K Q; M) = M_{f_{\geq A}, \eta^{-2}}(K Q; N) + O (\log (K / \varepsilon))$.
  We choose $\eta = \exp( - \exp(C_{5} \varepsilon^{-5}))$ with $C_{5} > 0$ an absolute constant much larger than $C_{4}$.
  We also assume that,
  $$
  \frac{1}{\delta^4}\exp(-2M_{f_{\geq A}, \eta^{-2}}(K Q; N) \leq \exp( - C_{6} \exp(\varepsilon^{-5}) )
  $$
  with $C_{6}$ much larger than $C_{4}$. This holds under the assumption that,
  $$
  M_{f_{\geq A, \eta^{-2}}}(K Q; N) \geq \exp(\exp(C_{7} \varepsilon^{-5}))
  $$
  for some large absolute constant $C_{7}$.
  Under these assumptions,
  \begin{equation}\label{eq:majarc_bound_final}
    \norm{(S_1 + S_2)\mathbf{1}_{\mathfrak{M} }}_2 \leq  \frac{\varepsilon}{8} \norm{S}_2
  \end{equation}
  Now, by Proposition A, collecting (\ref{eq:s3_upper}), (\ref{eq:s2_upper}), (\ref{eq:minarc_bd_final}), and (\ref{eq:majarc_bound_final}), we have that
  \begin{equation}\label{eq:main_thm_bd_together}
    C_{6} \cdot c \sqrt{Q} \varepsilon^{5} \leq  C_{6} \varepsilon^{5} \sqrt{c Q} \cdot N^{-1/2} \norm{S}_2\\ \le \int_0^1 |S(\alpha)|d\alpha\le \Delta_1
  \end{equation}
  with $C_{6}$ an absolute constant.
  Recalling that $\Delta_{1} = c_{2} \Delta / (\sqrt{c} \varepsilon)$, the above is a contradiction provided that,
 $$
 C_{7} \cdot c^{-3} \varepsilon^{-20} \Delta^{2} < Q.
 $$
 with $C_{7} > 0$ an absolute constant.
 We have thus shown that for any $\varepsilon > 0$ there exists absolute constants $K_{i} > 0$ such that the following are contradictory
 \begin{enumerate}
   \item We have,
         $$
         \sum_{p \in [Q, N^{1/100}]} \frac{1}{p} > \frac{64 \varepsilon^{-2}}{c}.
         $$
   \item We have $Q \leq N^{\delta} \exp(- K_{1} \varepsilon^{-5})$ with $\delta = \exp(-\exp(\exp(K_{2} \varepsilon^{-5})))$,
   \item We have, $K_{3} \cdot c^{-3} \varepsilon^{-20} \Delta^{2} \leq Q$,
   \item We have,
         \begin{equation}
           \label{eq:neg}
         |M_{f_{\geq A}, \eta^{-2}}(U Q; N)| \geq \exp(\exp(K_{4} \varepsilon^{-5})).
       \end{equation}
         with $A = \exp(K_{5} \varepsilon^{-4})$, $\eta = \exp(- \exp(K_{6} \varepsilon^{-5}))$ and $U = \exp(K_{7} \varepsilon^{-5})$
 \end{enumerate}

 We pick $Q = \lceil K_{3} c^{-3} \varepsilon^{-20} \Delta^{2} \rceil$ so that the third condition is satisfied.
 If we assume that $\Delta \leq N^{\gamma}$ with
 $$
 \gamma = \exp \Big (- \frac{1}{c} \exp(\exp(K_{8} \varepsilon^{-5})) \Big )
 $$
 with $K_{8} > 0$ a sufficiently large absolute constant then the first and second conditions are also satisfied.
 Therefore with this choice of $\gamma$ under the assumptions of the theorem and $\Delta \leq N^{\gamma}$, the negation of \eqref{eq:neg} must hold,
 that is,
 $$
 |M_{f_{\geq A}, \eta^{-2}}(U Q; N)| < \exp(\exp(K_{5} \varepsilon^{-5})).
 $$
 This implies that there exists a $t \in \mathbb{R}$ and a primitive character of conductor $q$ with
 $$
 (1 + |t|) q \ll c^{-3} \exp(\exp(K \varepsilon^{-5})) \Delta^{2}
 $$
 and $K > 0$ an absolute constant, such that,
 \begin{equation} \label{eq:bdd}
 \sum_{\Delta^{2}  / q \leq p \leq N} \frac{1 - \Re \overline{f(p)} \psi(p) p^{it}}{p} \leq \exp(\exp(K' \varepsilon^{-5})) + \sum_{p \in [\Delta^{2} / q, U Q / q]} \frac{1}{p}.
\end{equation}
It remains to note that the sum over $p$ is $\ll \varepsilon^{-20}$ provided that $\Delta > 1 / c$, and we are done.
%This implies that
%\begin{equation} \label{eq:concl}
%\sum_{\Delta^{2} / q \leq p \leq N} \frac{1 - \Re \overline{f(p)} \psi(p) p^{it}}{p}
%\ll \frac{1}{c} \exp(\exp(K' \varepsilon^{-5})).
%\end{equation}
%Stating the result with this conclusion allows us to remove the assumption that $\Delta \leq N^{\gamma}$ since \eqref{eq:concl}
%is vacuous if $\Delta > N^{\gamma}$.

% This means that for any given $\varepsilon > 0$, there exists constants $K_{1}$, $K_{2}$, $K_{3}$, $K_{4}$, $\eta$, $\delta$, all depending on $\varepsilon$, such that the
%  Since $\|S\|_{2} \geq \sqrt{c N}$, this is a contradiction.
 % We have therefore shown that for any $\varepsilon \in (0,1)$ there exists a choice of $K, A, \eta$ depending on $\varepsilon$ %such that the following conditions are incompatible,
 % $$
 % a
%  $$
%  This is a contradiction if $K_5(\eta)\exp(-M_{f_{\geq A}, \eta^{-2}}(K_{2} Q;N))\le c'/(20c)$ (recalling (\ref{eq:k1k2_lwr})), so it follows that
%  $M_{f_{\geq A}, \eta^{-2}}(KQ;N)\gg_{c, c'} 1$ as long as $KQ\le N^{\exp(-\eta^{-4})}$. The desired result follows upon noting that
%  since $Q\asymp_{c, c'}\Delta^2$, we have that $M_{f, \eta^{-2}}(\Delta^2; N) = M_{f_{\geq A}, \eta^{-2}}(KQ; N) + O_{c, c'}(1)\gg_{c, c'} 1$.
 % Thus, there exists some primitive $\psi$ of conductor $q\le Q$ and $|t|\le \eta^{-2}Q/q$ such that
 % \begin{equation*}
 %   \sum_{\Delta^2/q\le p\le N\\ (p, q) = 1}\frac{1 - \Re f(p)\conj{\psi(p)p^{it}}}{p}\ll_{c' ,c} 1,
 % \end{equation*}
 % so the desired result follows as long as $\Delta$ is at most a sufficiently small power of $N$ in terms of $c, c'$.

\end{proof}

\section{Orthogonality results}
The main purpose of this section is to collect a number of
results that allow to swap additive harmonics with multiplicative ones, that is,
$$
e \Bigl ( \frac{n a}{q} + n \theta \Bigr ) \leftrightarrow \chi(n) n^{it}.
$$
Such results are essentially well-known and can be proven by a
combination of Gallagher's Lemma and Mellin transforms.

\begin{lemma} \label{le:orthogonality:major}
  Let $a(n)$ be a sequence of complex numbers.
  Let $W$ be a smooth function compactly supported in $(1/4,4)$.
  Then, for all $N \geq 1$ and $X \geq 1$,
  $$
  \int_{|\theta| \leq 1 / X} \Big | \sum_{n} a(n) e(n \theta) W \Bigl ( \frac{n}{N} \Bigr ) \Big |^{2} d \theta
  \ll \frac{1}{N} \int_{\mathbb{R}} \Big | \sum_{n} a(n) n^{-it} W \Bigl ( \frac{n}{N} \Bigr ) \Big |^{2} \min \Bigl ( 1, \frac{(N / X)^{2}}{1 + |t|^{2}} \Bigr ) \,d t,
  $$
  with an absolute constant in the $\ll$.
\end{lemma}
\begin{proof}
  By Gallagher's Lemma (see Gallagher~\cite[Lemma 1]{Gallagher}),
  $$
  \int_{|\theta| \leq 1 / X} \Big | \sum_{n} a(n) W \Bigl ( \frac{n}{N} \Bigr ) e(n \theta) \Big |^{2} d \theta \ll \frac{1}{X^{2}} \int_{\mathbb{R}} \Big | \sum_{x \leq n \leq x + X/2} a(n) W \Bigl ( \frac{n}{N} \Bigr ) \Big |^{2} \,d x .
  $$
  On the other hand, by \cite[Lemma 2.2]{Soundararajan},
  \begin{align*}
    \frac{1}{X^{2}} \int_{\mathbb{R}} & \Big | \sum_{x \leq n \leq x + X/2} a(n) W \Bigl ( \frac{n}{N} \Bigr ) \Big |^{2} dx \\ & \ll \frac{1}{N} \int_{\mathbb{R}} \Big | \sum_{n} a(n) n^{-it} W \Bigl ( \frac{n}{N} \Bigr ) \Big |^{2}\cdot \min \Bigl ( 1, \frac{(N / X)^{2}}{1 + |t|^{2}} \Bigr ) dt.
  \end{align*}
\end{proof}

We also record the following simple lemma.
\begin{lemma} \label{le:orthogonality:minor}
  Let $a(n)$ be a sequence of complex numbers. Then,
  $$
  \sum_{(a,q) = 1} \Big | \sum_{n} a(n) e \Bigl ( \frac{n a}{q} \Bigr ) \Big |^{2} = \frac{1}{\varphi(q)} \sum_{\chi \pmod{q}} \Big | \sum_{n} a(n) c_{\chi}(n) \Big |^{2},
  $$
  where
  $$
  c_{\chi}(n) := \sum_{\substack{x \pmod{q} \\ (x,q) = 1}} \chi(x) e \Bigl ( \frac{n x}{q} \Bigr ).
  $$
\end{lemma}
\begin{proof}
  This follows from writing
  $$
  e \Bigl ( \frac{n a}{q} \Bigr ) = \frac{1}{\varphi(q)} \sum_{\chi \pmod{q}} \overline{\chi}(a) c_{\chi}(n),
  $$
  expanding the square, and executing orthogonality in $a$ over Dirichlet characters.
\end{proof}

Combining Lemma \ref{le:orthogonality:minor} and Lemma \ref{le:orthogonality:major} we obtain the following result.

\begin{lemma} \label{le:majorarcsgallagher}
  Let $1 \leq Q \leq N^{1/10}$, and define the major arcs
  $$
  \mathfrak{M}_{Q, N} := \bigcup_{\substack{(a,q) = 1 \\ q \leq Q}} \Bigl ( \frac{a}{q} - \frac{Q}{q N}, \frac{a}{q} + \frac{Q}{q N} \Bigr ).
  $$
  Let $a(n)$ be a sequence such that $|a(n)| \leq 1$.
  Let $W$ be a smooth compactly supported function in $(1/4, 4)$. Then, for any $A \geq 1$,
  \begin{align*}
    \int_{\mathfrak{M}_{Q, N}} & \Big | \sum_{n} a(n)e(n \alpha) W \Bigl ( \frac{n}{N} \Bigr ) \Big |^{2} d \alpha
    \\ & \ll \frac{1}{N} \sum_{q \leq Q} \frac{1}{\varphi(q)} \sum_{\chi \pmod{q}} \int_{|t| \leq A Q / q} \Big | \sum_{n} a(n) c_{\chi}(n) n^{-it} W \Bigl ( \frac{n}{N} \Bigr ) \Big |^{2} dt + \frac{N}{A^2} \cdot \| W \|_{2,0}^2
  \end{align*}
  with an absolute implicit constant in $\ll$.
\end{lemma}
\begin{proof}
  By the definition of the major arcs the integral is
  $$
  \leq \sum_{q \leq Q} \sum_{(a,q) = 1} \int_{|\theta| \leq Q / (q N)} \Big | \sum_{n} a(n) e \Bigl ( \frac{n a}{q} \Bigr ) e(n \theta) W \Bigl ( \frac{n}{N} \Bigr ) \Big |^{2} d \theta.
  $$
  By Lemma~\ref{le:orthogonality:minor} this is
  $$
  \sum_{q \leq Q} \frac{1}{\varphi(q)} \sum_{\chi \pmod{q}} \int_{|\theta| \leq Q / (q N)} \Big | \sum_{n} a(n) c_{\chi}(n) e(n \theta) W \Bigl ( \frac{n}{N} \Bigr ) \Big |^{2} d \theta.
  $$
  By Lemma~\ref{le:orthogonality:major} this is then
  $$
  \ll \frac{1}{N} \sum_{q \leq Q} \frac{1}{\varphi(q)} \sum_{\chi \pmod{q}} \int_{\mathbb{R}} \Big | \sum_{n} a(n) c_{\chi}(n) n^{-it}W \Bigl ( \frac{n}{N} \Bigr ) \Big |^{2} \min \Bigl (1 , \frac{(Q / q)^{2}}{1 + |t|^{2}} \Bigr ) d t.
  $$
  We now split the integral into a part with $|t| \leq A Q / q$ and the remaining part with $|t| > A Q / q$.
  It suffices to bound the latter part. To do so we split the integral into dyadic intervals. The contribution of each dyadic piece is bounded by
  $$
  \sum_{k \geq 0} \frac{1}{N} \sum_{q \leq Q} \frac{1}{\varphi(q)} \sum_{\chi \pmod{q}} \frac{1}{A^{2} 2^{2k}} \int_{|t| \leq 2^{k} Q} \Big | \sum_{n} a(n) c_{\chi}(n) n^{-it} W \Bigl ( \frac{n}{N} \Bigr ) \Big |^{2} dt.
  $$
  For $2^{k} Q^{3} \leq N^{4/5}$ we can use Lemma~\ref{le:orthogonality2} to see that to see that the contribution of $k$ with $2^{k} \leq N^{4/5} Q^{-3}$ is
  $$
  \ll \frac{1}{A^{2} N} \sum_{k \geq 0} 2^{-2k} \Bigl ( N \sum_{n} \Big | W \Big ( \frac{n}{N} \Big ) \Big |^{2} \Bigr ) \ll \frac{N}{A^{2}} \| W \|_{2,0}^2,
  $$
  as needed. On the other hand for $2^{k} Q^{3} > N^{4/5}$ we use the fact that
  $$
  \int_{|t| \leq 2^{k} Q} \Big | \sum_{n} a(n) W \Bigl ( \frac{n}{N} \Bigr ) \Big |^{2} dt \ll (2^{k} Q + N) N \| W \|_{2,0}^2,
  $$
  using the standard mean-value theorem for Dirichlet polynomials. Therefore
  the contribution of the $k$ with $2^{k} > N^{4/5} Q^{-3} \geq N^{1/2}$ is bounded by
  $$
  \frac{1}{N} \sum_{2^{k} \geq N^{1/2}} 2^{-2k} (2^{k} Q +  N) N Q \| W \|_{2,0}^2 \ll N^{5/6} \| W \|_{2,0}^2,
  $$
  which is negligible.

\end{proof}

\section{Mean-value theorems}
In this section we collect a number of mean-value theorems.
\begin{lemma} \label{le:granville}
  Let $a(n)$ be a sequence of complex number and $T \geq 1$ be given. Then,
  $$
  \int_{-T}^{T}  \Big | \sum_{T^{2} \leq n \leq x} \frac{a(n) \Lambda(n)}{n^{1 + it}} \Big |^{2} dt \ll \sum_{T^{2} \leq n \leq x} \frac{|a(n)|^{2} \Lambda(n)}{n}.
  $$
\end{lemma}
\begin{proof}
  This is \cite[Lemma 2.6]{NewPretProfGranvillec}.
\end{proof}

\begin{lemma} \label{le:orthogonality2}
  Let $a(n)$ be a sequence of complex coefficients. Let $N, Q, T > 0$ be such that $Q^{2} T \leq N^{1 - 1/100}$. Then,
  $$
  \sum_{q \leq Q} \frac{1}{\varphi(q)} \sum_{\chi \pmod{q}} \int_{|t| \leq T} \Big | \sum_{n \leq N} a(n)c_{\chi}(n) n^{-it} \Big |^{2} dt \ll N \sum_{n \leq N}|a(n)|^{2}.
  $$
\end{lemma}
\begin{proof}
  By duality it suffices to show that for integrable $\beta(\chi, t)$,
  $$
  \sum_{n} \Big | \sum_{q \leq Q} \sum_{\chi \pmod{q}} \int_{|t| \leq T} \beta(\chi, t) \frac{c_{\chi}(n)}{\sqrt{\varphi(q)}} n^{-it} dt \Big |^{2} \ll N \sum_{q \leq Q} \sum_{\chi \pmod{q}} \int_{|t| \leq T} |\beta(\chi, t)|^{2} \,d t .
  $$
  We put a smooth weight $\Phi(n / N)$ with $\Phi \geq 1$ for $n \leq N$ and $\Phi \geq 0$ on the sum over $n$. Opening the square we need to understand
  $$
  \sum_{n} \Phi \Bigl ( \frac{n}{N} \Bigr ) c_{\chi}(n) \overline{c_{\psi}(n)} n^{it - iu}
  $$
  for characters $\chi \pmod{q}$, $\psi \pmod{r}$ and $|t|, |u| \leq T$. Applying Poisson summation we see that the above is
  $$
  \frac{1}{q r} \sum_{\ell} \Bigl ( \sum_{x \pmod{q r}} c_{\chi}(x) \overline{c_{\psi}(x)} e \Bigl ( \frac{x \ell}{q r} \Bigr )  \Bigr ) \int_{\mathbb{R}} x^{it - iu} \Phi \Bigl ( \frac{x}{N} \Bigr ) e \Bigl ( - \frac{x \ell}{q r} \Bigr ) \,d x .
  $$
  Integrating by parts we see that the sum over $\ell$ is constrained to $\ell \ll N^{\varepsilon} \cdot q r \cdot (1 + |t - u|) / N$ up to an arbitrary power-saving error, and in particular,
  all the terms except for $\ell = 0$ are negligible. On the other hand at the central term we find
  $$
  \frac{N^{1 + it - i u}}{q r} \cdot \widetilde{\Phi}(1 + i (t - u)) \sum_{x \pmod{q r}} c_{\chi}(x) \overline{c_{\psi}}(x).
  $$
  Moreover, the sum over $x$ vanishes unless $\chi = \psi$, while in the case $\chi = \psi$ it is equal to $q^{2} \varphi(q)$ and the claim follows.
\end{proof}

The following is the classical hybrid large sieve.

\begin{lemma} \label{le:orthogonality3}
  Let $a(n)$ be a sequence of complex numbers. Then,
  $$  \sum_{q \leq Q} \sum_{\substack{\psi \pmod{q} \\ \text{primitive}}} \int_{|t| \leq T} \Big | \sum_{n \leq N} a(n) \psi(n) n^{-it} \Big |^2 dt \ll (Q^2 T + N) \sum_{n \leq N} |a(n)|^2.
  $$
\end{lemma}

From this we deduce the following result.

\begin{lemma} \label{le:fewbad}
  Let $a(n)$ be a sequence of complex numbers with $|a(n)| \leq 1$.
  Let $\mathcal{S}$ denote a collection of tuples $(t, \psi)$ with $\psi$ ranging over primitive characters of conductor $\leq Q$ and $t$ ranging over $\mathcal{T}_{\psi}$ a collection of well-spaced points in $[-T, T]$ which is allowed to depend on $\psi$. Let also $\mathcal{B} \subset \mathcal{S}$ denote the set of points $(t, \psi) \in \mathcal{S}$ at which
  $$
  \Big | \sum_{p \leq P} a(p) p^{-it} \psi(p) \Big | \geq \frac{P}{\Delta}.
  $$
  Then, $$
  |\mathcal{B}| \ll \log P \cdot k! \Delta^{2k},
  $$
  where $k$ is the smallest integer such that $P^{k} > Q^{2} T$.
\end{lemma}
\begin{proof}
  In order to prove the lemma pick the smallest integer $k \geq 1$ such that $P^{k} > Q^{2} T$. Then the cardinality of the set $\mathcal{B}$ is
  $$
  \leq \Bigl ( \frac{\Delta}{P} \Bigr )^{2k} \sum_{(t, \psi) \in \mathcal{B}} \Big | \sum_{p \leq P} a(p) p^{-it} \psi(p) \Big |^{2k}.
  $$
  We now relate this to a continuous integral by using subharmonicity: for every Dirichlet polynomial $D(s)$ of length $N$ we have
  $$
  |D(it)|^{2k} \leq \log^{2} P \int_{|\xi|, |\zeta| \leq 1 / \log P} |D(\xi + it + i \zeta)|^{2k} d\xi d \zeta.
  $$
  Thus, it suffices to bound
  $$
  \log P \int_{|\xi| < 1 / \log P}  \Bigl ( \frac{\Delta}{P} \Bigr )^{2k} \sum_{q \leq Q} \sum_{\chi \pmod{q}} \int_{|t| \leq 2T} \Big | \sum_{p \leq P} a(p) p^{-\xi -it} \psi(p) \Big |^{2k} dt.
  $$
  According to the previous lemma this is
  $$
  \ll \log P \cdot \Bigl ( \frac{\Delta}{P} \Bigr )^{2k} \cdot (Q^{2} T + P^{k}) k! P^{k}.
  $$
  Since $P^{k} > Q^{2} T$ this is
  $$
  \ll \log P \cdot k! \Delta^{2k}.
  $$
\end{proof}

Finally we need the following hybrid version of Halasz-Montgomery mean-value theorem.

\begin{lemma} \label{le:orthogonality4}
 Let $\mathcal{S}$ denote a set of tuples $(t, \psi_{i})$ with $\psi_{i}$ distinct primitive characters of conductor $\leq Q$ and $t \in \mathcal{T}_{\psi_i} \subset [-T, T]$ with $\mathcal{T}_{\psi}$ a set of well-spaced points (i.e $|t_i - t_j| \geq \tfrac 12$ for all $t_i, t_j \in \mathcal{T}_\psi$ with $i \neq j$). Assume that the $Q^{2} T \leq \sqrt{P}$ and $Q \leq P^{1/8}$. Then,
  $$
  \sum_{(t, \psi) \in \mathcal{S}} \Big | \sum_{P / 10 \leq p \leq 10 P} a(p) p^{-it} \psi(p) \Big |^{2} \ll \frac{1}{\log P} \sum_{P / 10 \leq p \leq 10 P} |a(p)|^{2}.
  $$
\end{lemma}
\begin{proof}
  By duality, it suffices to show that
  $$
  \sum_{P / 10 \leq p \leq 10 P} \Big | \sum_{(t, \psi) \in \mathcal{S}} \beta(t, \psi) p^{-it} \psi(p) \Big |^{2} \ll \frac{1}{\log P} \sum_{(t, \psi) \in \mathcal{S}} |\beta(t, \psi)|^{2}.
  $$
  To control for the fact that we have to deal with a sum over primes we introduce a sieve weight and add a smooth function $\Phi(n / P)$. Thus it suffices to bound
  $$
  \sum_{n} \Big | \sum_{(t, \psi) \in \mathcal{S}} \beta(t, \psi) n^{-it} \psi(n) \Big |^{2} \Phi \Bigl ( \frac{n}{P} \Bigr )\sum_{\substack{d  | n \\ d \leq z}} \lambda_{d}
  $$
  with $\Phi(x)\geq 1$ for $1/10 \leq x \leq 10$ and $\Phi(x) \geq 0$ everywhere else (and $\Phi$ compactly supported in $(0,\infty)$).
  Opening the square we get
  $$
  \sum_{\substack{(t_{0}, \psi_{0}) \in \mathcal{S}  \\ (t_{1}, \psi_{1}) \in \mathcal{S}}} \overline{\beta(t_{0}, \psi_{0})} \beta(t_{1}, \psi_{1}) \sum_{d \leq z} \lambda_{d} \cdot d^{it_{1} - it_{0}} \psi_{1}(d) \overline{\psi_{0}(d)} \sum_{n} n^{i t_{1} - i t_{0}} \psi_{1}(n) \overline{\psi_{0}(n)} \Phi \Bigl ( \frac{n d}{P} \Bigr ).
  $$
  We now execute the sum over $n$. Assuming that $\psi_{1}$ is of conductor $q_{1}$ and $\psi_{0}$ is of conductor $q_{0}$, we get
  $$
  \frac{1}{q_{0} q_{1}} \sum_{\ell} \Bigl ( \sum_{x \pmod{q_{0} q_{1}}} \psi_{1}(x) \overline{\psi_{0}(x)} e \Bigl ( \frac{x \ell}{q_{0} q_{1}} \Bigr )  \Bigr ) \int_{\mathbb{R}} x^{i t_{1} - i t_{0}} e \Bigl ( \frac{x \ell}{q_{1} q_{0}} \Bigr )\Phi \Bigl ( \frac{d x}{P} \Bigr ) \,d x .
  $$
  Integrating by parts we see that only the term $\ell = 0$ survive as long as $d q_{1} q_{0} \cdot (1 + |t_{0} - t_{1}|) < P^{3/4}$. Since we can certainly pick $z = P^{1/4}$, this will hold provided that $Q^{2} T \ll \sqrt{P}$.
  If the central term $\ell = 0$ is non-zero then $\psi_{1} = \psi_{0}$, and $q_{1} = q = q_{0}$. In that case the central term is equal to,
  $$
  \frac{\varphi(q)}{q} \int_{\mathbb{R}} x^{it_{1} - i t_{0}} \Phi \Bigl ( \frac{d x}{P} \Bigr ) dx = \frac{\varphi(q)}{q} \Bigl ( \frac{P}{d} \Bigr )^{1 - it_{0} + it_{1}} \cdot \widetilde{\Phi}(1 + i(t_{1} - t_{0})).
  $$
  Thus our main term is equal to
  $$
  \sum_{\substack{(t_{0}, \psi) \in \mathcal{S} \\ (t_{1}, \psi) \in \mathcal{S}}} \beta(t_{1}, \psi) \overline{\beta(t_{0}, \psi)} \cdot \widetilde{\Phi}(1 + i (t_{1} - t_{0})) \frac{\varphi(q_{\psi})}{q_{\psi}}\sum_{(d, q_{\psi}) = 1} \frac{\lambda_{d}}{d},
  $$
  where $q_{\psi}$ denotes the conductor of $\psi$. We notice that with a standard choice of sieve weight we have
  $$
  \frac{\varphi(q_{\psi})}{q_{\psi}} \sum_{(d, q_{\psi}) = 1} \frac{\lambda_{d}}{d} \ll \frac{1}{\log z} \asymp \frac{1}{\log P}
  $$
  as long as the conductors $q_{\psi} \leq \sqrt{z} = P^{1/8}$. We then get the result using the fact that the points $t_{i}$ are well-spaced and thus
  $$
  \sum_{\substack{(t_{0}, \psi) \in \mathcal{S} \\ (t_{1}, \psi) \in \mathcal{S}}} |\beta(t_{0}, \psi) \beta(t_{1}, \psi) \widetilde{\Phi}(1 + i (t_{1} - t_{0}))| \ll \sum_{(t, \psi) \in \mathcal{S}} |\beta(t, \psi)|^{2},
  $$
  as needed.
\end{proof}

\section{Bounds for Mellin transforms}

At various places we will need the following standard bound for the decay of the Mellin transform of a smooth function $W$.

\begin{lemma} \label{le:mellin_decay}
  Let $W : \mathbb{R} \rightarrow \mathbb{C}$ be a smooth function, compactly supported in $(0, \infty)$. Then,
  for any $a < b$, uniformly in $a < \Re s < b$, and $r \geq 1$,
  $$
  \widetilde{W}(s) := \int_{0}^{\infty} W(x) x^{s - 1} dx \ll_{a,b, r} \frac{\| W \|_{1, r}}{1 + |\Im s|^{r}}
  $$
\end{lemma}
\begin{proof}
  This is immediate by integration by parts.
\end{proof}

\section{Proof of Proposition A: The main criterion}
The purpose of this section is to prove Proposition A which we restate below.
\maincriterion
\begin{proof}
  Take $K > 0$ a parameter to be chosen.
  Let $\mathcal{E}_{K}\subset [0,1]$ be the subset of values $\alpha\in \mathfrak{m}$ for which
  \begin{equation}
    \label{eq:E_bound}
    |S_2(\alpha) + S_3(\alpha)| > \frac{K}{\Delta}N^{1/2}\norm{S}_2.
  \end{equation}
  Then, for all $\alpha \in \mathcal{E}_{K} \cap \mathfrak{m}$, by \ref{i:minarc_ass}, we have
  \begin{align*}
  |S(\alpha)| & \leq \frac{1}{\Delta}N^{1/2}\norm{S}_2 + |S_2(\alpha) + S_3(\alpha)| \\ & \leq (1 + K^{-1})(|S_2(\alpha) + S_3(\alpha)|) \leq (1 + K^{-1}) \cdot (|S_{2}(\alpha)| + |S_{3}(\alpha)|)
  \end{align*}
  Therefore, by \ref{i:s2_ass},
  \begin{equation} \label{eq:firstconclusion}
    \norm{S\mathbf{1}_{\mathfrak{m} \cap \mathcal{E}_{K}}}_2\le (1 + K^{-1})(\norm{S_2}_2 + \norm{S_3\mathbf{1}_{\mathfrak{m} }}_2)
    \le (1 + K^{-1})(\delta_2 + \norm{S_3\mathbf{1}_{\mathfrak{m} }}_2/\norm{S}_2)\norm{S}_2.
  \end{equation}
  On the other hand, by the triangle inequality and \ref{i:majarc_ass},
  \begin{align}
    \label{eq:secondconclusion}
    \norm{S\mathbf{1}_{\mathfrak{m}}}_2 & \geq (1 - \norm{(S_1 + S_2) \mathbf{1}_{\mathfrak{M}}}_2/\norm{S}_2 - \norm{S_3\mathbf{1}_{\mathfrak{M} }}_2/\norm{S}_2)\norm{S}_2
    \\ \nonumber & \geq (1 - \delta_1 -  \norm{S_3\mathbf{1}_{\mathfrak{M} }}_2/\norm{S}_2)\norm{S}_2.
  \end{align}
  Combining (\ref{eq:firstconclusion}) and (\ref{eq:secondconclusion}), and applying
  \ref{i:s3_ass}, we obtain that
  \begin{equation}
    \label{eq:thirdconclusion}
    \norm{S\mathbf{1}_{\mathfrak{m}\cap \mathcal{E}^{c}_{K} }}_2\ge (1 - (\delta_1 + \delta_2 + \delta_3) - K^{-1}(\delta_2 + \delta_3))\norm{S}_2,
  \end{equation}
  where $\mathcal{E}^{c}_{K} = [0,1] \backslash \mathcal{E}_{K}$.
  On the other hand, by \ref{i:minarc_ass} and (\ref{eq:E_bound}), we have that
  \begin{equation}\label{eq:l2l1_interp}
    \norm{S\mathbf{1}_{\mathfrak{m} \cap \mathcal{E}_{K}^{c} }}_2\le ((1 + K)\Delta^{-1}N^{1/2}\norm{S}_2)^{1/2} \norm{S}_1^{1/2}.
  \end{equation}
  Taking $K = 2(1 - (\delta_1 + \delta_2 + \delta_3))^{-1}$ and combining (\ref{eq:thirdconclusion}) and
  (\ref{eq:l2l1_interp}), we obtain the desired result.
\end{proof}

\section{Proof of Lemma C: Minor arc bounds}

Opening the definition of $c_{1}(n; I)$, we wish to bound
$$
\sum_{p \in I} \sum_{m} f(m p) W \Bigl ( \frac{m p}{N} \Bigr ) e(m p \alpha).
$$
where we omitted writting out a leading factor equal to $(\sum_{p \in I} p^{-1})^{-1}$.
We notice that the contribution of the integers on which $f(n p) \neq f(n) f(p)$
is $\ll N / Q^{3/4}$. We then open $W$ into a Mellin transform and split the sum
over $p$ into dyadic intervals, thus getting a bound of 
$$
\ll  \sum_{\substack{P M \asymp N \\ Q \leq P \leq N^{1/100}}} \int_{\mathbb{R}} |\widetilde{W}(i u)| \Big |  \sum_{m} f(m) m^{-iu} K \Bigl ( \frac{m}{M} \Bigr ) \sum_{p} f(p) K \Bigl ( \frac{p}{P} \Bigr ) p^{-iu} e(m p \alpha) \Big |,
$$
where $P$ and $N$ run over powers of two and
where $K$ is a smooth function compactly supported in $(1/4,4)$ such that
$$
\sum_{N} K \Bigl ( \frac{n}{N} \Bigr ) = 1
$$
for all integers $n \geq 1$ and with $N$ running over powers of two.
Finally it remains to H\"older on the sum over $p$. This yields a bound of 
\begin{equation*}
  M^{3/4} \Bigl ( \sum_{m \asymp M} \Big | \sum_{p \asymp P} \alpha_{p} e(m p \alpha) \Big |^{4} \Bigr )^{1/4}
\end{equation*}
for some coefficients $|\alpha_{p}| \leq 1$ supported on primes.
Finally we recall that $\alpha \in \mathfrak{m}$. We notice that every real $\alpha$ has a rational approximation with
$$
\Big | \alpha - \frac{a}{q} \Big | \leq \frac{1}{q (N / Q)}
$$
and $q \leq N / Q$. Since $\alpha \not \in \mathfrak{M}_{Q, N}$ we know that in addition $q > Q$. Thus we can write
$$
\alpha = \frac{a}{q} + \theta
$$
with $(a,q) = 1$, $Q \leq q \leq N / Q$ and $|\theta| \leq 1 / N$. Introducing a smooth function $V$, we conclude that,
\begin{multline*}
  \sum_{m \asymp M} \Big | \sum_{p \asymp P} \alpha_{p} e(m p \alpha) \Big |^{2} \\
  \leq \sum_{p_{1}, p_{2}, p_{3}, p_{4} \asymp P} \alpha_{p_{1}} \alpha_{p_{2}} \overline{\alpha_{p_{3}}} \overline{\alpha_{p_{4}}} \sum_{m} e \Bigl (m (p_{1} + p_{2} - p_{3} - p_{4}) \Bigl ( \frac{a}{q} + \theta \Bigr ) \Bigr ) V \Bigl ( \frac{m}{M} \Bigr ).
\end{multline*}
When $p_{1} + p_{2} \equiv p_{3} + p_{4} \pmod{q}$ we can bound the sum by
$$
\ll M \cdot \Bigl ( \frac{P^4}{(\log P)^4 \cdot Q} \Bigr ).
$$
We note that the entire point of taking a fourth power in Holder (instead of e.g the usual second power) is that
it leads to a more efficient bound in the diagonal since there are more variables present.
For the off-diagonal terms we notice that by Poisson summation,
$$
\sum_{m} e \Big( m (p_{1} + p_{2} - p_{3} - p_{4}) \Bigl ( \frac{a}{q} + \theta \Bigr ) \Bigr ) W \Bigl ( \frac{m}{M} \Bigr ) \ll M^{-A}
$$
since the dual sum is of length $\ll N^{o(1)} Q / M$ which is negligible.
Collecting these bounds we get a final bound of
$$
\| W \|_{1, 2} \Bigl ( \sum_{p \in [Q, N^{1/100}]} \frac{1}{p} \Bigr )^{-1} \frac{N}{\sqrt{Q}} \sum_{Q \leq 2^{k} \leq N^{1/100}} \frac{1}{k} + \frac{N}{Q^{3/4}} \ll \| W \|_{1, 2} \cdot \frac{N}{\sqrt{Q}},
$$
as claimed (the factor $(\sum_{p \in I} p^{-1})^{-1}$ comes from the definition of the sequence $c_1(\cdot)$).

\section{Proof of Lemma D: Reduction to completely multiplicative functions}
The purpose of this section is to prove Lemma D. We restate it below for convenience.
\reduction
\begin{proof}
  Let $f^{\star}$ be a multiplicative function such that
  $$
  f(n) = \sum_{d | n} f^{\star}(d).
  $$
  Let $f^{\vee}$ and $f^{\wedge}$ be defined by
  $$
  f^{\vee}(p^{\alpha}) = \begin{cases}
    f(p^{\alpha}) & \text{ if } p \leq A, \\
    1             & \text{ if } p > A
  \end{cases} \ , \ f^{\wedge}(p^{\alpha}) =
  \begin{cases}
    f(p^{\alpha}) & \text{ if } p > A,   \\
    1             & \text{ if } p \leq A.
  \end{cases}
  $$
  We can then write
  \begin{align*}
    f(n) & = f^{\vee}(n) f^{\wedge}(n) = \sum_{\substack{d | n \\ p | d \implies p \leq A}} f^{\star}(d) \cdot f^{\wedge}(n)  \\ & = \Bigl ( \sum_{\substack{d | n \\ p | d \implies p \leq A \\ \Omega(d) \leq 100 \log\log A}} f^{\star}(d) + \sum_{\substack{d | n \\ p | d \implies p \leq A \\ \Omega(d) > 100 \log\log A}} f^{\star}(d) \Bigr ) (f_{\geq A}(n) + (f^{\wedge} - f_{\geq A})(n)).
  \end{align*}
  Notice that
  $$
  \Big |        \sum_{\substack{d | n \\ p | d \implies p \leq A \\ \Omega(d) > 100 \log\log A}} f^{\star}(d) f^{\wedge}(n) \Big | \leq \sum_{\substack{d | n \\ p | d \implies p \leq A \\ \Omega(d) > 100 \log\log A}} 1,
  $$
  since $|f^{\star}(n)| \leq 1$ and $|f^{\wedge}(n)| \leq 1$. Furthermore, for any integer $M \geq \exp(A)$,
  $$
  \sum_{n \leq M} \Big | \sum_{\substack{d | n \\ p | d \implies p \leq A \\ \Omega(d) > 100 \log\log A}} 1 \Big |^{2} \cdot \Big | W \Big ( \frac{m}{M} \Big ) \Big |^2 \ll \frac{M}{\log A} \cdot \| W \|_{2,2}^2
  $$
  Similarly we notice that
  \begin{align*}
  \sum_{n \leq M} \Big | \sum_{\substack{d | n \\ p | d \implies p \leq A \\ \Omega(d) \leq 100 \log\log A}} 1 \Big|^{2} & \cdot |f^{\wedge}(n) - f_{\geq A}(n)|^{2} \cdot \Big | W \Big ( \frac{m}{M} \Big ) \Big |^2 \\ & \ll M \log A \sum_{p > A} \frac{1}{p^{2}} \cdot \| W \|_{2,2}^2 \ll \frac{N}{A} \cdot \| W \|_{2,2}^2
  \end{align*}
  Therefore,
  \begin{align*}
    \int_{\mathfrak{M}_{Q, N}} & \Big | \sum_{n} f(n) e(n \alpha) W \Bigl ( \frac{n}{N} \Bigr ) \Big |^{2} d \alpha \\ & \ll \int_{\mathfrak{M}_{Q, N}} \Big | \sum_{n} \sum_{\substack{d | n \\ p | d \implies p \leq A \\ \Omega(d) \leq 100 \log\log A}} f^{\star}(d) f_{\geq A}(n) e(n \alpha) W \Bigl ( \frac{n}{N} \Bigr ) \Big |^{2} d \alpha + \frac{N}{\log A} \cdot \| W \|_{2,2}^2.
  \end{align*}
  Notice that we have the trivial bound
  $$
  \sum_{\substack{p | d \implies p \leq A \\ \Omega(d) \leq 100 \log\log A}} 1 \leq \sum_{k \leq 100 \log\log A} \binom{A}{k} \ll K,
  $$
  and that the largest $d$ such that $p | d \implies p \leq A$ and $\Omega(d) \leq 100 \log\log A$ is $K$.
  Therefore by Cauchy-Schwarz,
  \begin{align*}
    \Big | \sum_{n} \sum_{\substack{d | n \\ p | d \implies p \leq A \\ \Omega(d) \leq 100 \log\log A}} & f^{\star}(d) f^{\wedge}(n) e(n \alpha) W \Bigl ( \frac{n}{N} \Bigr ) \Big |^{2} \\ & \leq K^{2} \sup_{d \leq K} \int_{\mathfrak{M}_{Q, N}} \Big | \sum_{\substack{d | n}} f_{\geq A}(n) e(n \alpha) W \Bigl ( \frac{n}{N} \Bigr ) \Big |^{2} d \alpha.
  \end{align*}
  It remains to show that
  \begin{equation} \label{eq:finaltrivial}
    \int_{\mathfrak{M}_{Q, N}} \Big | \sum_{\substack{d | n}} f_{\geq A}(n) e(n \alpha) W \Bigl ( \frac{n}{N} \Bigr ) \Big |^{2} d \alpha \leq \int_{\mathfrak{M}_{d Q, N}} \Big | \sum_{n} f_{\geq A}(n) e(n \alpha) W_{0} \Bigl ( \frac{n}{N} \Bigr ) \Big |^{2} d \alpha.
  \end{equation}
  We notice that
  $$
  \sum_{n} f_{\geq A}(d n) e(n d \alpha) W \Bigl ( \frac{d n}{N} \Bigr )  = \frac{f_{\geq A}(d)}{2\pi i}\int_{\mathbb{R}} \widetilde{W}(it) N^{it} d^{-it} \sum_{n} f_{\geq A}(n) e(d n \alpha) n^{-it} W_{0} \Bigl ( \frac{n}{N} \Bigr ) dt.
  $$
  Therefore, \eqref{eq:finaltrivial} is
  $$
  \ll \int_{\mathbb{R}} |\widetilde{W}(it)| \int_{\mathfrak{M}_{Q, N}} \Big | \sum_{n} f_{\geq A}(n) n^{-it} e(d n \alpha) W_{0} \Bigl ( \frac{n}{N} \Bigr ) \Big |^{2} d \alpha.
  $$
  Finally by a change of variable,
  $$
  \int_{\mathfrak{M}_{Q, N}} \Big | \sum_{n} f_{\geq A}(n) n^{-it} e(d n \alpha) W_{0} \Bigl ( \frac{n}{N} \Bigr ) \Big | d \alpha
  \leq \frac{1}{d} \int_{\mathfrak{M}_{d Q, N}} \Big | \sum_{n} f_{\geq A}(n) n^{-it} e(n \alpha) W_{0} \Bigl ( \frac{n}{N} \Bigr )\Big |^{2} d \alpha.
  $$
  Finally we can truncate the integral over $t$ to $|t| \leq A$ (since the part with $|t| > A$ contributes $\ll_{r} \| W \|_{1, r + 1} N A^{-r}$) and the claim follows. Note that the integral over $|t| \leq A$ contributes $\ll \| W\|_{1, 2} \ll \| W \|_{2,2}^2$.
\end{proof}

\section{Proof of Proposition E: Bound for integral over major arcs}
In this section we prove Proposition E. Given a primitive character $\psi$ of conductor $q$ we define
\cpsidef
The main input in our proof will be the following proposition.
\halaszvariant
\begin{proof}
  See \S\ref{se:halaszvariant}.
\end{proof}
Optimizing in $\varepsilon_{0}$ yields a final bound that is at most as strong as a saving of
$$
(\log (3 + M_{f, \psi, T}(R; N))^{-1}.
$$
over the trivial bound of $N^{2}$.
Using Ramar{\'e}'s identity instead of the Turan-Kubilius inequality it should be possible to improve this to a
saving of the form
$$
\exp( - c M_{f, \psi, T}(R; N) )
$$
for some absolute constant $c > 0$. Using Ramar{\'e}'s identity entails slightly more messy combinatorics and for this
reason we prefer to work with the Turan-Kubilius inequality. The dependence of the implicit constant $\ll$ on $W$ can be made explicit. Leaving this dependence implicit will not create issues for us, we will apply this Lemma only to a single fixed smooth function resulting from a partition of unity.

With this proposition in hand we are now ready to prove Proposition E. We reproduce its
statement below for convenience.
\smallmajorarcs
\begin{proof}
  Let $J = [N^{\eta}, N^{1/10}]$. By Lemma B, we have
  \begin{align*}
    \int_{\mathfrak{M}_{Q, N}} & \Big | \sum_{n} f(n) e(n \alpha) W \Bigl ( \frac{n}{N} \Bigr ) \Big |^{2} d \alpha \\ & \leq 2 \int_{\mathfrak{M}_{Q, N}} \Big | \sum_{n} f(n) c_{1}(n; J) e(n \alpha) W \Bigl ( \frac{n}{N} \Bigr ) \Big |^{2} d \alpha + 8 N \Bigl ( \sum_{p \in J} \frac{1}{p} \Bigr )^{-1}.
  \end{align*}
  Furthermore by Lemma~\ref{le:majorarcsgallagher}
 we have
  \begin{align} \nonumber
    \int_{\mathfrak{M}_{Q, N}} & \Big | \sum_{n} f(n) c_{1}(n; J) e(n \alpha) W \Bigl ( \frac{n}{N} \Bigr ) \Big |^{2} d \alpha \\ \label{eq:tofactor}  & \ll \frac{1}{N} \sum_{q \leq Q} \frac{1}{\varphi(q)} \sum_{\chi \pmod{q}} \int_{|t| \leq \eta^{-1} Q / q} \Big | \sum_{n} f(n)c_{1}(n; J) c_{\chi}(n) n^{-it} W \Bigl ( \frac{n}{N} \Bigr ) \Big |^{2} dt.
  \end{align}
  We notice that the weight $c_{1}(n; J)$ restricts the sum over $n$ to integers that can be written as $m p$ with $p \in J$. We can furthermore assume that $(m,p) = 1$: using Lemma~\ref{le:orthogonality2} we can bound the contribution of the integers with $p | m$ by $\ll N^{1 - \eta}$. Note that for $(m,p) = 1$ we have $c_{\chi}(m p) = c_{\chi}(m) \overline{\chi}(p)$. Moreover, introducing a partition of unity, i.e a smooth $K$ compactly supported in $(0, \infty)$ such that
  $$
  \sum_{k \geq 0} K \Bigl ( \frac{n}{2^{k}} \Bigr )  = 1 \ , \ \forall n \in \mathbb{N},
  $$
  we can express the Dirichlet polynomial over $n$ in \eqref{eq:tofactor} as
  $$
  \Bigl ( \sum_{p \in J} \frac{1}{p} \Bigr )^{-1} \sum_{\substack{M P \in N I \\ N^{\eta} \leq P \leq N^{1/10}}}
  \sum_{p} K \Bigl ( \frac{p}{P} \Bigr ) f(p) \overline{\chi(p)} p^{-it} \sum_{\substack{(m, p) = 1}} K \Bigl ( \frac{m}{M} \Bigr ) f(m) c_{\chi}(m) m^{-it} W \Bigl ( \frac{m p}{N} \Bigr )
  $$
  where $I \subset (0, \infty)$ is a closed interval such that the support of $W$ is contained in $I$ and
  where both $M$ and $P$ run over powers of two greater or equal to one. Using Lemma~\ref{le:orthogonality2}, we can drop the condition $(m,p) = 1$ at a total cost that is $\ll \| W \|_{2, 0}^2 \cdot N^{1 - \eta} \log^{100} N$. Furthermore, opening $W$ into a Mellin transform and applying Cauchy-Schwarz, we see that \eqref{eq:tofactor} is
  \begin{align} \label{eq:finalbound}
    \ll \int_{\mathbb{R}} & |\widetilde{W}(iu)| \Bigl ( \sum_{p \in J} \frac{1}{p} \Bigr )^{-2} \cdot
                            \log N                                                                                              \\ \nonumber &  \sum_{\substack{M P \in N I \\ N^{\eta} \leq P \leq N^{1/10}}} \frac{1}{N} \sum_{q \leq Q} \frac{1}{\varphi(q)} \int_{|t| \leq \eta^{-1} Q / q} | \mathcal{N}(it + iu, \chi; M)|^{2} \cdot |\mathcal{P}(it + iu, \chi; P)|^{2} dt \,d u,
  \end{align}
  where
  $$
  \mathcal{N}(s, \chi; M) := \sum_{m} f(n) c_{\chi}(m) m^{-s} K \Bigl ( \frac{m}{M} \Bigr )
  $$
  and
  $$
  \mathcal{P}(s, \chi; P) := \sum_{p} f(p) \overline{\chi}(p) p^{-s} K \Bigl ( \frac{p}{P} \Bigr ).
  $$
  Notice that due to the rapid decay of $\widetilde{W}(iu)$ we can truncate $u$ at $|u| \leq \eta^{-1}$ after discarding an error term that is $\ll_{r}  \| W \|_{1, r} N \eta^{r - 1}$ for any given $r > 1$.
  This estimate is obtained by applying a trivial bound on $\mathcal{P}$ (i.e only counting the number of primes $\asymp P$) and using a mean-value theorem to bound the
  remaining Dirichlet polynomial $\mathcal{N}$. Notice that if we then enlarge the integration over $|t|$ to $\eta^{-2} Q / q$ the integration over $u$  becomes redundant,
  and we can assume without loss of generality that $u = 0$.

  We now focus our attention on the inner sum over $q \leq Q$ in \eqref{eq:finalbound}, fixing the dyadic scale $P$ and $M$. For any $\chi \pmod{q}$ induced by a
  primitive character $\psi$ we have $\mathcal{P}(s, \chi; P) = \mathcal{P}(s, \psi; P) + O(\omega(q; N^{\eta}, N^{1/10}))$ on $\Re s \geq 0$ and where $\omega(q; A, B)$ counts the number of distinct prime
  factors of $q$ in the interval $[A, B]$. By an application of Lemma~\ref{le:orthogonality3} (note that by assumption $N^{\eta} \leq P \leq Q$) we see that the total error
  induced by $O(\omega(q;N^{\eta},N^{1/10}))$ is $\ll N^{1 - \eta /2 }$.
  Thus, we are left with upper bounding the following expression:
  $$
  \sum_{q \leq Q} \sum_{\substack{\psi \pmod{q} \\ \text{primitive}}} \int_{|t| \leq \eta^{-2} Q / q} \Bigl ( \sum_{\substack{\chi \pmod{r} \\ \psi \text{ induces } \chi \\ r \leq Q}} \frac{1}{\varphi(r)} |\mathcal{N}(it, \chi; M)|^{2} \Bigr ) |\mathcal{P}(it, \psi; P)|^{2} \,d t ,
  $$
  since our earlier expression had an integration with a cut off at $|t| \leq \eta^{-2} Q / r$ (note that $q \leq r$ so the integration over $t$ is enlarged).
  We further bound the above integral by the following sum\footnote{strictly speaking two sums: we pick a sequence of maxima $0 \leq t_{1} \leq 1 \leq t_{2} \leq 2 \leq ...$ in each unit interval, we then obtain two sums, one over $t_{2j + 1}$, and another one over $t_{2j}$, within each sum we then have $|t_{2j + 1} - t_{2k + 1}| \geq 1$ if $j \neq k$ and $|t_{2j} - t_{2k}| \geq 1$ if $j \neq k$} over well-spaced points $\mathcal{T}_{\psi} \subset [- \eta^{-2} Q / q, \eta^{-2} Q / q]$:
  \begin{equation} \label{eq:inner}
    \sum_{q \leq Q} \sum_{\substack{\psi \pmod{q} \\ \text{primitive}}} \sum_{t \in \mathcal{T}_{\psi}} \Bigl ( \sum_{\substack{\chi \pmod{r} \\ \psi \text{ induces } \chi \\ r \leq Q}} \frac{1}{\varphi(r)} |\mathcal{N}(it, \chi; M)|^{2} \Bigr ) |\mathcal{P}(it, \psi; P)|^{2}.
  \end{equation}
  Denote by $\mathcal{S}$ the set of tuples $(t, \psi)$ appearing in the sum above, with $\psi$ a primitive character of modulus $q$ and $t \in \mathcal{T}_{\psi}$.  We now separate the set of tuples $(t, \psi)$ into two categories.
  The good tuples
  $$
  \mathcal{G} := \Big \{ (t, \psi) \in \mathcal{S} : |\mathcal{P}(it, \psi; P)| \leq P (\log P)^{-100}  \Big \},
  $$
  and the remaining bad tuples
  $$
  \mathcal{B} := \mathcal{S} \backslash \mathcal{G}.
  $$
  We then split \eqref{eq:inner} into a sums over good and bad tuples, that is
  \begin{equation} \label{eq:split}
    \Bigl ( \sum_{(t, \psi) \in \mathcal{G}} + \sum_{(t, \psi) \in \mathcal{B}} \Bigr ) \Bigl ( \sum_{\substack{\chi \pmod{r} \\ \psi \text{ induces } \chi \\ r \leq Q}} \frac{1}{\varphi(r)} |\mathcal{N}(it, \chi; M)|^{2} \Bigr ) |\mathcal{P}(it, \psi; P)|^{2}.
  \end{equation}
  On the good tuples, we use the definition of $\mathcal{G}$ to bound $\mathcal{P}$ and Lemma~\ref{le:orthogonality2} to bound the rest. We thus find that the sum over the good tuples in \eqref{eq:split} is
  $$
  \ll \frac{P^{2}}{\log^{200} P} \cdot M^{2} \ll \frac{1}{\eta^{200}} \frac{N^{2}}{\log^{200} N}
  $$
  since $P > N^{\eta}$.
  To bound the contribution of the bad tuples in \eqref{eq:split} we use Proposition G to first bound point-wise the sum over induced characters. This shows that the sum over $(t, \psi) \in \mathcal{B}$ in \eqref{eq:split} is
  $$
  \ll \Bigl ( \frac{1}{\delta^{3}} \sup_{\substack{\psi \pmod{q} \\ \psi \text{ primitive } \\ q \leq Q}} \exp \Bigl ( - 2 M_{f, \psi, \eta^{-2} Q / q}(Q; M) \Bigr ) + \eta^{4} \Bigr ) \cdot M^{2} \sum_{(t, \psi) \in \mathcal{B}} |\mathcal{P}(it, \psi; P)|^{2}
  $$
  since $Q \leq N^{\delta}$ and $\delta = \exp( - \eta^{-4} )$.
  We notice incidentally that $M_{f, \psi, T}(R; M) = M_{f, \psi, T}(R; N) + O(1)$ since $M \gg N^{1/2}$ always.
  We now use Lemma~\ref{le:fewbad} and Lemma \ref{le:orthogonality4} to conclude that the sum above is dominated by essentially $O(1)$ large values, and thus, the above is
  $$
  \ll M^{2} \cdot \frac{P^{2}}{\log^{2} P} \cdot \Bigl ( \frac{1}{\delta^{3}} \exp \Bigl ( - 2 M_{f, \eta^{-2}}(Q, N) \Bigr ) + \eta^{4} \Bigr ).
  $$
  Thus \eqref{eq:inner} is
  $$
  \ll \frac{N^{2}}{\eta^{2} \log^{2} N} \cdot \Bigl ( \frac{1}{\delta^{3}} \exp \Bigl ( - 2 M_{f,\eta^{-2}}(Q, N) \Bigr ) + \eta^{4} \Bigr ),
  $$
  and we conclude that \eqref{eq:finalbound} is
  $$
  \ll \| W \|_{1,2} \cdot \Big ( \frac{N}{\delta^{4}} \exp \Bigl ( - 2 M_{f, \eta^{-2}}(Q, N) \Bigr ) + \eta N \Big ),
  $$
  as needed (note that $\| W \|_{1,2} \ll \| W \|_{2,2}^2$ thanks to the $+1$ in our definition of $\| W \|_{p,r}$).

\end{proof}

\section{Proof of Proposition G: A large sieve variant of Halasz's theorem} \label{se:halaszvariant}
We restate here Proposition G.
Recall that for a primitive character $\psi$ of conductor $q$ we define
\cpsidef
We then have the following result.
\halaszvariant
We will need the following simple mean-value theorem.

\begin{lemma} \label{le:bound}
  Let $\psi$ be a primitive character of modulus $q$.
  Let $a(n)$ be a sequence of arbitrary complex coefficients. Then, for $R^{2} \leq N^{1/10}$,
  $$
  \sum_{\substack{\chi \pmod{r} \\ \psi \text{ induces } \chi \\ r \leq R}} \frac{1}{\varphi(r)} \Big |\sum_{n \leq N} a(n) c_{\chi}(n) \Big |^{2} \ll
  N \sum_{n \leq N} |a(n)|^{2}.
  $$
\end{lemma}
\begin{proof}
  By duality it suffices to show that
  $$
  \sum_{n \leq N} \Big | \sum_{\substack{\chi \pmod{r} \\ \psi \text{ induces } \chi \\ r \leq R}} \beta(\chi) \cdot \frac{c_{\chi}(n)}{\sqrt{\varphi(r)}} \Big |^{2} \leq N \sum_{\substack{\chi \pmod{r} \\ \psi \text{ induces } \chi \\ r \leq R}} |\beta(\chi)|^{2}.
  $$
  Putting a smooth function on the sum over $n$ and expanding the square we get
  $$
  \sum_{\substack{\chi_{1} \pmod{r_{1}} \\ \chi_{2} \pmod{r_{2}} \\ \psi \text{ induces } \chi_{1}, \chi_{2} \\ r_{1}, r_{2} \leq R}} \beta(\chi_{1}) \overline{\beta(\chi_{2})} \sum_{n} \frac{c_{\chi_{1}}(n)}{\sqrt{\varphi(r_{1})}} \cdot \frac{\overline{c_{\chi_{2}}(n)}}{\sqrt{\varphi(r_{2})}} \Phi \Bigl ( \frac{n}{N} \Bigr ).
  $$
  Applying Poisson summation to modulus $r_{1} r_{2}$ we see that the terms with $\chi_{1} \neq \chi_{2}$ are completely negligible since
  $$
  \sum_{x \pmod{r_{1} r_{2}}} c_{\chi_{1}}(x) \overline{c_{\chi_{2}}(x)} = 0
  $$
  for $r_{1} \neq r_{2}$ and $R^2 \leq N^{1/10}$.
  Meanwhile,
  $$
  \sum_{n} \frac{1}{\varphi(r)} |c_{\chi}(r)|^{2} \Phi \Bigl ( \frac{n}{N} \Bigr ) = \frac{\widehat{\Phi}(0) N}{r} \cdot \frac{1}{\varphi(r)}\sum_{x \pmod{r}} |c_{\chi}(x)|^{2},
  $$
  and the sum over $x$ is equal to $r \varphi(r)$.
\end{proof}
Set $\varepsilon^{2} := \varepsilon_{0}$ so that $R \leq N^{\varepsilon_{0}}$ becomes $R \leq N^{\varepsilon^{2}}$.
By a variant of the Turan-Kubilius inequality and Lemma~\ref{le:bound} we see that \eqref{eq:main} is
\begin{equation} \label{eq:startpoint}
  \ll \frac{1}{S_{1}^{2} S_{2}^{2}} \sum_{\substack{\chi \pmod{r} \\ \psi \text{ induces } \chi \\ r \leq R}} \frac{1}{\varphi(r)} \Bigl | \sum_{n} f(n) \Bigl ( \sum_{\substack{N^{\varepsilon} \leq p \leq N^{1/10} \\ p | n}} 1 \Bigr ) \Bigl ( \sum_{\substack{N^{\varepsilon^{2}} \leq p \leq N^{\varepsilon} \\ p | n}} 1 \Bigr ) c_{\chi}(n) W \Bigl ( \frac{n}{N} \Bigr ) \Bigr |^{2} + \frac{N^{2}}{\log (1 / \varepsilon)},
\end{equation}
where
$$
S_{1} := \sum_{N^{\varepsilon^{2}} \leq p \leq N^{\varepsilon}} \frac{1}{p} \ , \ S_{2} := \sum_{N^{\varepsilon} \leq p \leq N^{1/10}} \frac{1}{p}.
$$
We write, for brevity,
$$
w(n) := \Bigl ( \sum_{\substack{N^{\varepsilon} \leq p \leq N^{1/10} \\ p | n}} 1 \Bigr )\Bigl ( \sum_{\substack{N^{\varepsilon^{2}} \leq p \leq N^{\varepsilon} \\ p | n}} 1 \Bigr ).
$$
Furthermore, by \cite[Lemma 5.4]{MV}, we have for $\chi \pmod{r}$ induced by $\psi \pmod{q}$ that if $q | r / (r, n)$,
$$
c_{\chi}(n) = \overline{\psi} \Bigl ( \frac{n}{(r, n)} \Bigr ) \frac{\varphi(r)}{\varphi(r / (r, n))} \mu \Bigl ( \frac{r / (r,n)}{q} \Bigr ) \psi \Bigl ( \frac{r / (r,n)}{q} \Bigr ) \tau(\psi).
$$
and $c_{\chi}(n) = 0$ if $q$ does not divide $r / (r, n)$.
Therefore, splitting the sum according to the value of $(r, n)$, we get
\begin{align}
  \nonumber & \sum_{n} f(n) w(n) c_{\chi}(n) W \Bigl ( \frac{n}{N} \Bigr )
  \\ \nonumber & = \sum_{\substack{d | r  \\ q | r / d}} \frac{f(d)\varphi(r)}{\varphi(r / d)} \mu \Bigl ( \frac{r / d}{q} \Bigr ) \psi \Bigl ( \frac{r / d}{q} \Bigr )
  \tau(\psi) \sum_{\substack{(n, r / d) = 1}} f(n) \overline{\psi(n)} w(n) W \Bigl ( \frac{n d}{N} \Bigr )
  \\ \nonumber & = \tau(\psi) \sum_{\substack{\ell | r / d \\ d | r \\ q | r / d}} \mu(\ell) f(\ell) \overline{\psi(\ell)} \frac{f(d) \varphi(r)}{\varphi(r / d)} \mu \Bigl ( \frac{r / d}{q} \Bigr ) \psi \Bigl ( \frac{r / d}{q} \Bigr ) \sum_{n} f(n) \overline{\psi(n)} w(n) W \Bigl ( \frac{n d \ell}{N} \Bigr ),
\end{align}
where in the last line we opened the condition $(n, r / d) = 1$ using Mobius inversion, and where we used that $w(d \ell n) = w(n)$ since $w$ depends only on prime factors larger than $R$, while both $d$ and $\ell$ are $\leq R$.
We now open $W$ into a Mellin transform. This allows us to rewrite this entire expression as
\begin{align*}
  \frac{N \tau(\psi)}{2\pi} \int_{\mathbb{R}} & \Bigl ( \sum_{\substack{p | n \implies p \leq N}} \frac{f(n) \overline{\psi(n)}}{n^{1 + it}} \Bigr ) \cdot \Bigl ( \sum_{N^{\varepsilon} \leq p \leq N^{1/10}} \frac{f(p) \overline{\psi(p)}}{p^{1 + it}} \Bigr ) \cdot  \Bigl ( \sum_{N^{\varepsilon^{2}} \leq p \leq N^{\varepsilon}} \frac{f(p) \overline{\psi(p)}}{p^{1 + it}} \Bigr ) \\ & \times \Bigl ( \sum_{\substack{\ell | r / d \\ d | r \\ q | r / d}} \mu(\ell) f(\ell) \overline{\psi(\ell)} \frac{f(d) \varphi(r)}{\varphi(r / d)} \mu \Bigl ( \frac{r / d}{q} \Big  ) \psi \Bigl ( \frac{r / d}{q} \Bigr ) \cdot (d \ell)^{-1-it} \Bigr ) \cdot N^{it} \widetilde{W}(1 + it) dt.
\end{align*}
We rewrite this as
$$
\frac{N \tau(\psi)}{2\pi} \int_{\mathbb{R}} \mathcal{F}(1 + it ; \overline{\psi}) \mathcal{P}_{1}(1 + it; \overline{\psi}) \mathcal{P}_{2}(1 + it; \overline{\psi}) g_{t,\psi}(r) N^{it} \widetilde{W}(1 + it) \,d t,
$$
where
$$
g_{t, \psi}(r) := \sum_{\substack{\ell | r / d \\ d | r \\ q | r / d}} \mu(\ell) f(\ell) \overline{\psi(\ell)} \frac{f(d) \varphi(r)}{\varphi(r / d)} \mu \Bigl ( \frac{r / d}{q} \Big  ) \psi \Bigl ( \frac{r / d}{q} \Bigr ) \cdot (d \ell)^{-1-it},
$$
and
$$
\mathcal{P}_{1}(s;\psi) := \sum_{N^{\varepsilon^{2}} \leq p \leq N^{\varepsilon}} \frac{f(p) \psi(p)}{p^{s}} \ , \ \mathcal{P}_{2}(s;\psi) := \sum_{N^{\varepsilon} \leq p \leq N^{1/10}} \frac{f(p)\psi(p)}{p^{s}},
$$
and finally,
$$
\mathcal{F}(s;\psi) := \sum_{p |n \implies p \leq N} \frac{f(n) \psi(n)}{n^{s}}.
$$
Importantly, we notice that the function $g_{t, \psi}(r)$ is essentially multiplicative in nature. This is the content of the following lemma.
\begin{lemma}
  We have
  $$
  g_{t, \psi}(r) = \mathbf{1}_{q | r} \cdot h_{t, \psi} \Bigl ( \frac{r}{q} \Bigr ),
  $$
  where $h_{t, \psi}$ is a multiplicative function such that
  $$
  h_{t, \psi}(p) = f(p) p^{-it} - \psi(p) + O(1 / p),
  $$
  and $h_{t, \psi}(p^{\alpha}) = O(1)$ for all $\alpha > 1$ with a uniform constant in the $O(1)$.
\end{lemma}
\begin{proof}
  In the definition of $g_{t, \psi}(r)$, we make a change of variables $d \mapsto r / d$. This allows us to re-write $g_{t, \psi}(r)$ as
  $$
  \sum_{q | d | r} \Bigl ( \sum_{\ell | d} \mu(\ell) f(\ell)\overline{\psi(\ell)} \ell^{-1-it} \Bigr ) \frac{f(r / d) \varphi(r)}{\varphi(d) (r / d)^{1 + it}} \cdot \mu \Bigl ( \frac{d}{q} \Bigr ) \psi \Bigl ( \frac{d}{q} \Bigr )
  $$
  Clearly this implies $q | r$. Writting $\kappa = d / q$ and $u = r / q$ we rewrite the above as
  $$
  \mathbf{1}_{q | r} \cdot \sum_{\kappa | u} \phi_{t, \psi}(\kappa q) \frac{f(u / \kappa) \varphi(u q)}{\varphi(\kappa q) (u / \kappa)^{1 + it}} \mu(\kappa) \psi(\kappa),	$$
  where
  $$
  \phi_{t, \psi}(v) := \sum_{\ell | v} \mu(\ell) f(\ell) \overline{\psi(\ell)} \ell^{-1 -it}.
  $$
  We further write $u = v w$ with $w = \prod_{p | (u,q)} p^{v_{p}(u)}$ and $v = \prod_{p | u, (p,q) = 1} p^{v_{p}(u)}$ where, as usual, $v_{p}(u)$ is maximal such that $p^{v_{p}(u)}$ divides $u$.
  Thus $(v, w) = 1$ and we can also write any $\kappa | u$ as $\kappa = \nu_0 \omega_0$ with $\nu_0 | v$ and $\omega_0 | w$. In this circumstance, we find that
  $$
  \frac{\varphi(u q)}{\varphi(\kappa q)} =\frac{\varphi(v)w}{\varphi(\nu_0)\omega_0}  \ , \ \phi_{t, \psi}(\kappa q) = \phi_{t, \psi}(\nu_0) \phi_{t, \psi}(q).
  $$
  We notice that since $\psi$ is a primitive character of conductor $q$, we have $\phi_{t, \psi}(q) = 1$.
  Thus,
  $$
  g_{t, \psi}(r)  = \mathbf{1}_{q | r} \cdot \sum_{\nu_0 | v} \frac{\varphi(v)}{\varphi(\nu_0)} \phi_{t, \psi}(\nu_0) \mu(\nu_0) \psi(\nu_0) \frac{f(v / \nu_0)}{(v / \nu_0)^{1 + it}} \sum_{\omega_0 | w} \frac{f(w / \omega_0)}{(w / \omega_0)^{it}} \mu(\omega_0) \psi(\omega_0).
  $$
  Notice also that since $\psi$ is a primitive character of modulus $q$ and $\omega$ consists only of prime factors dividing $q$, we have
  $$
  \sum_{\omega_0 | w} \frac{f(w / \omega_0)}{(w / \omega_0)^{it}} \mu(\omega_0) \psi(\omega_0) = \frac{f(w)}{w^{it}}.
  $$
  The result now follows from reading off the factorization prime by prime.
\end{proof}
We can thus bound the left-hand side of \eqref{eq:startpoint} by
$$
\frac{N^{2} q}{S_{1}^{2} S_{2}^{2}} \sum_{\substack{r \leq R\\ q | r}} \frac{1}{\varphi(r)} \Big | \int_{\mathbb{R}} \mathcal{F}(1 + it, \overline{\psi})\mathcal{P}_{1}(1 + it, \overline{\psi}) \mathcal{P}_{2}(1 + it, \overline{\psi}) h_{t} \Bigl ( \frac{r}{q} \Bigr ) N^{it} \widetilde{W}(1 + it) dt \Big |^{2}.
$$
where $q$ came from $|\tau(\psi)|^{2} = q$.
We write $r = q r'$ with $r' \leq R / q$ and note that $\varphi(q r') \geq \varphi(q) \varphi(r')$. Furthermore we bound the above sum by extending the sum over $r'$ to a sum over all integers such that $p | r' \implies p \leq R / q$. We then expand the square and bound the above expression by
\begin{align} \label{eq:tobound}
  \frac{q}{\varphi(q)} \frac{N^{2}}{S_{1}^{2} S_{2}^{2}} \cdot \int_{\mathbb{R}^{2}} & |\widetilde{W}(1 + it) \widetilde{W}(1 + iu)| |(\mathcal{P}_{1} \mathcal{P}_{2}) (1 + iu; \overline{\psi})| \cdot | (\mathcal{P}_{1} \mathcal{P}_{2})(1 + iv; \overline{\psi})| \\ \nonumber & \times |\mathcal{F}(1 + iu; \overline{\psi}) \overline{\mathcal{F}(1 + iv; \overline{\psi})}| \cdot
                                                                                                                                                                                                                                                                                      \Big | \sum_{p | r' \implies p \leq R / q} \frac{h_{u, \psi}(r') \overline{h_{v, \psi}(r')}}{\varphi(r')} \Big | du dv
\end{align}
We notice that
\begin{align*}
  & \Big | \mathcal{F}(1 + i u; \overline{\psi}) \overline{\mathcal{F}(1 + i v; \overline{\psi})} \sum_{p | r' \implies p \leq R / q} \frac{h_{u, \psi}(r') \overline{h_{v, \psi}(r')}}{\varphi(r')} \Big |
  \\ & \asymp \exp \Bigl ( \Re \sum_{\substack{p \leq N}} \frac{f(p) \overline{\psi(p)} p^{- i u} + \overline{f(p)} \psi(p) p^{iv}}{p} + \Re \sum_{\substack{p \leq R / q}} \frac{(f(p) p^{-iu} - \psi(p)) \cdot (\overline{f(p)} p^{iv} - \overline{\psi(p)})}{p} \Bigr )
  \\ & \asymp \exp \Bigl ( \Re \sum_{\substack{p \leq R / q}} \frac{|f(p)|^{2} p^{i v - i u} + \mathbf{1}_{(p,q) = 1}}{p} + \Re \sum_{R / q \leq p \leq N} \frac{f(p)\overline{\psi(p)}p^{-iu}}{p} + \Re \sum_{R / q \leq p \leq N} \frac{\overline{f(p)}\psi(p)p^{-iv}}{p} \Bigr ).
\end{align*}
The supremum of this expression over $|u|,|v| \leq T$ is then bounded by
$$
\ll \frac{\varphi(q)}{q}\sup_{|t| \leq T} \exp \Bigl ( \sum_{\substack{p \leq R / q}} \frac{2}{p} + 2 \sum_{R / q \leq p \leq N} \frac{\Re f(p) \overline{\psi(p)} p^{it}}{p}
+ \sum_{R/q\le p\le N\\ p | q}\frac{1}{p} \Bigr ).
$$
Notice also that the supremum over all $u,v \in \mathbb{R}$ is also bounded by $\ll \log^{2} N$.
Thus \eqref{eq:tobound} is
\begin{align*}
  \ll & \frac{N^{2}}{S_{1}^{2}S_{2}^{2}} \frac{\varphi(q)}{q}\sup_{|t| \leq T} \exp \Bigl ( \sum_{p \leq R / q} \frac{2}{p}
        + 2 \sum_{R / q \leq p \leq N} \frac{\Re f(p)\overline{\psi(p)}p^{it}}{p} + \sum_{R/q\le p\le N\\ p | q}\frac{1}{p}\Bigr ) \\
      & \times \Bigl ( \int_{|u| \leq T} |\mathcal{P}_{1}(1 + i u; \overline{\psi}) \mathcal{P}_{2}(1 + i u; \overline{\psi})| |\widetilde{W}(1 + iu)| du \Bigr )^{2} + \mathcal{E}_{T},
\end{align*}
where $\mathcal{E}_{T}$ is bounded by
\begin{align*}
  \frac{N^{2}}{S_{1}^{2}S_{2}^{2}} & \cdot \log^{2} N \int_{\mathbb{R}} |\mathcal{P}_{1}(1 + i u; \overline{\psi}) \mathcal{P}_{2}(1 + i u; \overline{\psi})| |\widetilde{W}(1 + iu)| \,d u  \\
                                   & \times \int_{|u| \geq T} |\mathcal{P}_{1}(1 + i u; \overline{\psi}) \mathcal{P}_{2}(1 + i u; \overline{\psi})| |\widetilde{W}(1 + iu)| \,d u.
\end{align*}
To bound the integral we now use Cauchy-Schwarz and notice that by Lemma~\ref{le:granville}, for $T \leq \log N$,
$$
\int_{|u| \leq T} |\mathcal{P}_{1}(1 + iu; \overline{\psi})|^{2}  du \ll \sum_{N^{\varepsilon} \leq p \leq N^{1/10}} \frac{1}{p \log p} \ll \frac{S_{2}}{\varepsilon \log N}.
$$
The important point here is that Lemma~\ref{le:granville} is a mean-value theorem specifically for Dirichlet polynomials supported on primes:
it gains one logarithm compared to the usual mean-value theorem valid for any Dirichlet polynomial. Similarly,
$$
\int_{|u| \leq T} |\mathcal{P}_{2}(1 + iu; \overline{\psi})|^{2} du \ll \sum_{N^{\varepsilon^{2}} \leq p \leq N^{\varepsilon}} \frac{1}{p \log p} \ll \frac{S_{1}}{\varepsilon^{2} \log N}.
$$
Thus,
$$\Bigl ( \int_{|u| \leq T} |\mathcal{P}_{1}(1 + i u; \overline{\psi}) \mathcal{P}_{2}(1 + i u; \overline{\psi})| du \Bigr )^{2} \ll \frac{1}{\varepsilon^{3}} \cdot \frac{S_{1} S_{2}}{\log^{2} N}.$$
This means that the final bound for \eqref{eq:tobound} that we obtain is
$$
\ll \frac{N^{2}}{\varepsilon^{3}} \sup_{|t| \leq T} \exp \Bigl ( 2 \sum_{R / q \leq p \leq N \\ (p, q) = 1} \frac{\Re f(p)\overline{\psi(p)}p^{it} - 1}{p} \Bigr )
+ \frac{N^{2}}{\log(1/\varepsilon)}  + \mathcal{E}_{T}
$$
for any given $A > 10$ and $T > 10$. We also notice that by Lemma~\ref{le:granville} and a dyadic dissection, we have
$$\Bigl ( \int_{|u| > T} |\mathcal{P}_{1}(1 + i u; \overline{\psi}) \mathcal{P}_{2}(1 + i u; \overline{\psi})| |\widetilde{W}(1 + iu)| du \Bigr )^{2}
\ll_{A} \frac{1}{\varepsilon^{3}} \cdot \frac{S_{1} S_{2}}{T^{A} \log^{2} N}$$
for any $A > 10$. Thus, if we choose $T = 1 / \varepsilon_{0}$, we will find that
$$
\mathcal{E}_{T} \ll N^{2} \varepsilon_{0},
$$
which is entirely sufficient.
This concludes the proof. Note that we could have obtained a better bound by using Ramare's identity instead of Turan Kubilius, as it would have lead to a saving of $\varepsilon N^{2}$ in place of $N^{2} / \log(1 / \varepsilon)$.

\section{Pretentious multiplicative functions}

Throughout this section, our main tool will be the following result originally due to Gallagher.
\begin{proposition} \label{le:galaba}
  Let $q \geq 1$ be an integer and $t \in \mathbb{R}$. Then,
  uniformly in $q (1 + |t|) \leq Q \leq N$ we have, for non-quadratic $\chi$,
  $$
  \Big | \sum_{N \leq p \leq 2N} \frac{\chi(p) p^{it}}{p} \Big | \ll \Big ( \exp \Big ( - \frac{c \log N}{\log Q} \Big ) + \frac{1}{\log N} \Big ) \sum_{N \leq p \leq 2N} \frac{1}{p},
  $$
  with $c > 0$ an absolute constant.  Moreover for quadratic $\chi$ we have,
  \begin{equation} \label{eq:bon}
  \sum_{N \leq p \leq 2N} \frac{\chi(p)}{p} \leq \Big ( \exp \Big ( - \frac{c \log N}{\log Q} \Big ) + \frac{C}{\log N} \Big )\sum_{N \leq p \leq 2N} \frac{1}{p}
  \end{equation}
  with $c, C > 0$ absolute constants.
\end{proposition}
\begin{proof}
  See for example \cite[Theorem 5.13]{IK} and integrate by parts. We point out that when $\chi$ has a Siegel zero \eqref{eq:bon} is negative and this account for the one-sided inequality in our conclusion.
  \end{proof}

Given $1$-bounded multiplicative functions $f,g$, and an interval $I$, define,
$$
\mathbb{D}(f, g; I)^{2} = \sum_{p \in I} \frac{1 - \Re f(p)\overline{g(p)}}{p}.
$$
For any $1$-bounded multiplicative functions $f,g,h$ we have the triangle inequality,
$$
\mathbb{D}(f, g; I) \leq \mathbb{D}(f, h; I) + \mathbb{D}(h, g; I).
$$
The following Lemma shows that if the pretentious distance $\mathbb{D}(1, f, g)$ with $g(n) = \chi(n)n^{it}$ is bounded and $f$ is real-valued then $\chi$ needs to be a quadratic character.
\begin{lemma}
  Let $f : \mathbb{N} \rightarrow [-1,1]$ be a multiplicative function. Let $\chi$ be a non-principal, non-quadratic Dirichlet character of conductor $\leq Q$. Suppose that $|t| \leq Q$ and $I = [Q, N]$ with $N \geq Q^{A}$.
  Let $g(n) = \chi(n) n^{it}$.
  Then, for all $A$ and $N$ sufficiently large,
  $$
  c \log A \leq \mathbb{D}(f, g; I)^{2}
  $$
  with $c$ an absolute constant.
\end{lemma}
\begin{proof}
  We notice that,
  \begin{align*}
  \mathbb{D}(f, g; I)^{2} & = \sum_{p \in I} \frac{1 - \Re f(p) \overline{\chi(p)} p^{-it}}{p} = \sum_{p \in I} \frac{1 - f(p) \Re \overline{\chi}(p) p^{-it}}{p}
  \end{align*}
  using that $f$ is real-valued.
  Let $J = [Q^{\sqrt{A}}, N]$. Then, the above is
  \begin{equation} \label{eq:as}
  \geq \sum_{p \in J} \frac{1 - f(p) \Re \overline{\chi(p)} p^{-it}}{p}.
\end{equation}
We note that for $|x| \leq 2$,
$$
|x| \leq h(x) := 1 + \frac{1}{2} \cdot (x^{2} - 1) - \frac{1}{18} \cdot (x^{2} - 1)^{2}.
$$
Therefore, we have that \eqref{eq:as} is at least,
$$
\sum_{p \in J} \frac{1 - h(\Re \chi(p) p^{it})}{p}.
$$
Using the previous Lemma this gives that \eqref{eq:as} is at least
$$
\Big (\frac{13}{48} - \exp( - c \sqrt{A}) - \frac{C}{\log N} \Big ) \sum_{p \in J} \frac{1}{p}
$$
with $c, C > 0$ absolute constant.
Since,
$$
\sum_{p \in J} \frac{1}{p} \gg \log A
$$
the result follows, provided that $A$ and $N$ are sufficiently large.
\end{proof}

\begin{lemma}
    Let $f : \mathbb{N} \rightarrow [-1,1]$ be a multiplicative function. Let $g(n) = n^{it}$.
    Let $I = [Q, N]$ with $N \geq Q^{A}$ and $|t| \leq Q$.
    Then, for all $A$ and $N$ sufficiently large, and $|t| \geq 2 / \log N$,
    $$
    c \log \min(A, |t| \log N) \leq \mathbb{D}(f, g; I)^{2}
    $$
    with $c > 0$ an absolute constant.
  \end{lemma}
  \begin{proof}
    If $|t| \geq 1$ we let $J = [Q^{\sqrt{A}}, N]$ and notice that,
    \begin{align*}
      \mathbb{D}(f, g; I)^{2} & \geq \sum_{p \in J} \frac{1 - f(p) \Re p^{it}}{p} \geq \sum_{p \in J} \frac{1 - |\Re p^{it}|}{p} \\
                              & \geq \sum_{p \in J} \frac{1 - h(\Re p^{it})}{p} = \Big ( \frac{13}{48} - \exp ( -c \sqrt{A}) - \frac{C}{\log N} \Big ) \sum_{p \in J} \frac{1}{p}.
    \end{align*}
    with $c, C > 0$ absolute constants.
    On the other hand, when $|t| \leq 1$ we can evaluate the sum by integration by parts.
    In that case we pick $J = [e^{1/|t|}, N]$ we get a lower bound that is,
    $$
    \geq \sum_{p \in J} \frac{1 - |\cos(t \log p)|}{p} \geq c \sum_{p \in J} \frac{1}{p} \geq \frac{c}{2} \cdot \log (|t| \log N)
    $$
    with $c > 0$ an absolute constant, for all sufficiently large $N$.
  \end{proof}
  We then have the following result.
  \begin{lemma} \label{le:hastobereal}
    Let $f : \mathbb{N} \rightarrow [-1,1]$ be a multplicative function.
    Let $\psi$ denote a primitive character of conductor $q$, let $t \in \mathbb{R}$ and $g(n) = \psi(n) n^{it}$.
    Let $I = [Q, N]$ with $N \geq Q^{A}$ and $q (1 + |t|) \leq Q$.
    If,
    $$
    \mathbb{D}(f,g; I)^{2} \leq c
    $$
    with $c > 1$
    and
    $A$ is sufficiently large with respect to $c$
    then $\psi$ is a quadratic character, and
    $$
    \mathbb{D}(f, \psi; I)^{2} \leq K c.
    $$
    with $K > 0$ an absolute constant.
  \end{lemma}
  \begin{proof}
    If $\psi$ is not quadratic, nor principal, then by the first Lemma,
    $$
    c' \log A \leq \mathbb{D}(f, g; I)^{2} \leq c,
    $$
    with $c' > 0$ an absolute constant.
    This is a contradiction for all sufficiently large $A$.
    Therefore $\psi$ has to be a quadratic character. Let $f_{1}(n) = f(n) \psi(n)$ and $g_{1}(n) = n^{it}$.
    If $|t| > 2 / \log N$,
    $$
    c' \log \min(A, |t| \log N) \leq \mathbb{D}(f_{1}, g_{1}; I)^{2} = \mathbb{D}(f, g; I)^{2} \leq c.
    $$
    This is a contradiction if $A$ is sufficiently large and $|t| \geq C / \log N$ with $C = 2 c / c'$. Thus, for all $A$ sufficiently large we are left with the possibility that $|t| \leq C / \log N$. We observe that,
    $$
    \mathbb{D}(f_{1}, g_{1}; I)^{2} = \mathbb{D}(f_{1}, 1; I)^{2} + O(1) = \mathbb{D}(f, \psi; I)^{2} + O(1),
    $$
    using the Taylor expansion,
    $$
    n^{it} = 1 + O \Big ( C \frac{\log n}{\log N} \Big )
    $$
    and the classical estimate,
    $$
    \sum_{p \leq N} \frac{\log p}{p \log N} \ll 1.
    $$
    It follows that,
    $$
    \mathbb{D}(f, \psi; I)^{2} \leq K c.
    $$
    with $K > 0$ an absolute constant.
  \end{proof}

Finally we show that the pretentious distance of two distinct real characters cannot be small.
\begin{lemma} \label{le:twoquad}
  Let $\chi$ and $\psi$ be two quadratic characters of conductor $\leq Q$ such that $\chi \psi$ is not principal. Let
  $I = [Q, N]$ with $N \geq Q^{A}$. Then for all $A$ and $N$ sufficiently large,
  $$
  c \log A \leq \mathbb{D}(\chi, \psi; I)^{2}
  $$
  with $c > 0$ an absolute constant.
\end{lemma}
\begin{proof}
  Appealing to Proposition \ref{le:galaba} and integrating by parts we see that if $\chi \psi$ is not principal, then,
  $$
  \mathbb{D}(\chi, \psi; I)^{2} \geq \Big (1 - \exp ( - c \sqrt{A}) - \frac{C}{\log N} \Big )\sum_{p \in I} \frac{1}{p}
  $$
  Thus the claim follows for all sufficiently large $A$ and $N$.
\end{proof}

\section{Proof of Corollary \ref{cor:first}}

By assumption, for any $\varepsilon > 0$ and all sufficiently large $N > N_{0}(\varepsilon)$,
\begin{equation} \label{eq:asumpt}
\int_{0}^{1} \Big | \sum_{n \leq N} f(n) e(n \alpha) \Big | d \alpha \leq N^{\varepsilon^{2}}.
\end{equation}
Select a sequence of $N$ called $N_{1}$, $N_{2}$, $\ldots$ such that $N_{i - 1} = N_{i}^{\varepsilon^{2}}$
and $N_{0}(\varepsilon) < N_{1}$. Using our Main Theorem A and a little of ``pretentious theory'' we have
the following Lemma.
\begin{lemma}
  Let $f$ be a $1$-bounded multiplicative function such that,
  \begin{equation} \label{eq:liminf}
  \liminf_{N \rightarrow \infty} \frac{1}{N} \sum_{n \leq N} |f(n)|^{2} > \rho.
  \end{equation}
  with $\rho > 0$.
  Let $\varepsilon > 0$ be such that \eqref{eq:asumpt} holds for all $N > N_{0}(\varepsilon)$ sufficiently
  large.
  Then there exists an absolute constant $c > 0$ such that for all $\varepsilon > 0$ sufficiently small, and all $N > M_{0}(\varepsilon, \rho)$ there exists a real quadratic character $\chi$ such that
  $$
  \sum_{N^{2 \varepsilon^{2}} \leq p \leq N} \frac{1 - f(p) \chi(p)}{p} \leq c.
  $$
  and the conductor of $\chi$ is $\ll_{\rho} N^{2 \varepsilon^{2}}$.
\end{lemma}
\begin{remark} Note that \eqref{eq:liminf} implies that the mean-value of $|f|^2$ exists. \end{remark}
\begin{proof}
  By our Main Theorem B, there exists a $t \in \mathbb{R}$ and a primitive Dirichlet character $\psi$ of conductor $q$
  with $(1 + |t|)q \ll_{\rho} N^{2 \varepsilon^{2}}$ such that,
  $$
  \sum_{N^{2 \varepsilon^{2}} \leq p \leq N} \frac{1 - \Re f(p) \overline{\psi}(p) p^{-it}}{p} \leq c
  $$
  with $c > 0$ an absolute constant. Given $I = [N^{\varepsilon^{2}}; N]$, we now introduce the distance function,
  $$
  \mathbb{D}(f,g; I)^{2} = \sum_{p \in I} \frac{1 - \Re \overline{f}(p) g(p)}{p}.
  $$
  By Lemma \ref{le:hastobereal}, once $\varepsilon$ is sufficiently small this forces the character $\psi$ to be quadratic and implies that,
  $$
  \mathbb{D}(f, \psi; I)^{2} \ll C
  $$
  with $C > 0$ an absolute constant. The claim follows.
  \end{proof}

Therefore to each scale $[N_{i - 1}, N_{i}]$ we can associate a Dirichlet character $\chi_{i}$ such that
\begin{equation} \label{eq:asumpt1}
\sum_{N_{i - 1} \leq p \leq N_{i}} \frac{1 - f(p) \chi_{i}(p)}{p} \leq c.
\end{equation}
We introduce a new scale $[M_{i - 1}, M_{i}]$ with
\begin{equation} \label{eq:def1}
M_{i - 1} = N_{i-1}^{1 / \varepsilon} = N_{i}^{\varepsilon} \ , \ M_{i} = N_{i}^{1 / \varepsilon}
\end{equation}
which intersects both $[N_{i - 1}, N_{i}]$ and $[N_{i}, N_{i + 1}]$.
Appealing again to the above Lemma to each scale $[M_{i - 1}, M_{i}]$ we can
associate a quadratic Dirichlet character $\psi_{i}$ of conductor $\ll_{\rho} N_{i}^{2 \varepsilon}$
and such that,
\begin{equation} \label{eq:asumpt2}
\sum_{M_{i - 1} \leq p \leq M_{i}} \frac{1 - f(p) \psi_{i}(p)}{p} \leq c.
\end{equation}
The character $\psi_{i}$ turns to be equal to both $\chi_{i}$ and $\chi_{i + 1}$.
\begin{lemma}
  Suppose that $\varepsilon > 0$ is sufficiently small.
  Suppose that \eqref{eq:asumpt1} and \eqref{eq:asumpt2} holds with $M_{i}$ and $N_{i}$ related
  by \eqref{eq:def1}. Then, $\chi_{i} = \psi_{i} = \chi_{i + 1}$.
\end{lemma}
\begin{proof}
  Let $I = [N_{i}^{\varepsilon}, N_i]$. Notice that by \eqref{eq:asumpt1} and \eqref{eq:asumpt2} we have,
  $$
  \mathbb{D}(f, \chi_{i}; I)^{2} \leq c
  $$
  and
  $$
  \mathbb{D}(f, \psi_{i}; I)^{2} \leq c.
  $$
  By the triangle inequality, this implies that,
  $$
  \mathbb{D}(\psi_{i}, \chi_{i}; I) \leq \mathbb{D}(\psi_{i}, f; I) + \mathbb{D}(f, \chi_{i}; I) \leq 2 \sqrt{c}.
  $$
  However by Lemma \ref{le:twoquad}, if $\chi_{i} \neq \psi_{i}$ then the left-hand side is at least
  $$
  \geq \sqrt{ c \log \frac{1}{\varepsilon}}.
  $$
  This is a contradiction for all $\varepsilon > 0$ sufficiently small, and in particular $\chi_{i} = \psi_{i}$. To conclude that $\chi_{i} = \psi_{i + 1}$ we repeat the same argument but with a different choice of interval $I$. We pick $I = [N_i, N_i^{1 / \varepsilon}]$. Then we have,
  $$
  \mathbb{D}(f, \psi_{i}; I)^{2} \leq c
  $$
  but we also have,
  $$
  \mathbb{D}(f, \chi_{i + 1}; I)^{2} \leq c.
  $$
  Therefore, if $\psi_{i}$ and $\chi_{i + 1}$ differ, then,
  $$
  \sqrt{c \log \frac{1}{\varepsilon}} \leq \mathbb{D}(\psi_{i}, \chi_{i + 1}; I) \leq \mathbb{D}(\psi_{i}, f; I) + \mathbb{D}(f, \chi_{i + 1}; I) \leq 2 \sqrt{c}
  $$
  and this is a contradiction for all sufficiently small $\varepsilon > 0$.
\end{proof}
Thus all the $\chi_{i}$ have to be equal and we conclude that for every $\varepsilon$ there exists
a quadratic character $\chi$ of conductor $\ll_{\varepsilon} 1$ such that, for all $i \geq 1$
$$
\sum_{N_{i} \leq p \leq N_{i + 1}} \frac{1 - f(p) \chi(p)}{p} \ll 1.
$$
Summing over all $i$ we conclude that for all $N \geq N_{0}(\varepsilon)$
$$
\sum_{p \leq N} \frac{1 - f(p)\chi(p)}{p} \ll \varepsilon^{2} \log\log N.
$$
Thus for every $\varepsilon > 0$ sufficiently small, there exists a quadratic character $\chi_{\varepsilon}$ of conductor $\leq C(\varepsilon)$ such that
$$
\sum_{p \leq N} \frac{1 - f(p)\chi(p)}{p} \ll \varepsilon^{2} \log\log N.
$$
for all $N \geq N_{0}(\varepsilon)$.
We claim that all these quadratic characters are equal once $\varepsilon$ is sufficiently small. Let $I = [1, N]$. Indeed, for any $\varepsilon_{1}$ and $\varepsilon_{2}$, sufficiently small, by the triangle inequality
$$
\mathbb{D}(\psi_{\varepsilon_{1}}, \psi_{\varepsilon_{2}}; I)^2 \leq \mathbb{D}(f, \psi_{\varepsilon_{1}}; I) + \mathbb{D}(f, \psi_{\varepsilon_{2}}; I) \leq 2 \max(\varepsilon_{1}, \varepsilon_{2}) \sqrt{\log\log N}
$$
and all $N \geq N_{0}(\varepsilon_{1}, \varepsilon_{2})$. On the other hand, it is easy to see that if $\psi_{\varepsilon_{1}} \neq \psi_{\varepsilon_{2}}$ then
$$
\mathbb{D}(\psi_{\varepsilon_{1}}, \psi_{\varepsilon_{2}}; I) \geq \log\log N + O_{\varepsilon_{1}, \varepsilon_{2}}(1).
$$
as $N \rightarrow \infty$.
This is a contradiction for all $\varepsilon_{1}, \varepsilon_{2}$ sufficiently small.
Thus all the characters $\psi_{\varepsilon}$ are equal for all $\varepsilon > 0$ sufficiently small. This concludes the proof.
\bibliography{l2}
\bibliographystyle{plain}

\end{document}